\documentclass[graybox]{svmult}

\usepackage{mathptmx}       % selects Times Roman as basic font
\usepackage{helvet}         % selects Helvetica as sans-serif font
\usepackage{courier}        % selects Courier as typewriter font
\usepackage{type1cm}        % activate if the above 3 fonts are
\usepackage{makeidx}         % allows index generation
\usepackage{graphicx}        % standard LaTeX graphics tool
\usepackage{multicol}        % used for the two-column index
\usepackage[bottom]{footmisc}% places footnotes at page bottom
\usepackage[english]{babel}
\usepackage[latin1]{inputenc}
\usepackage{latexsym}
\usepackage{epsfig}
\usepackage{graphicx}
\usepackage{color}
\usepackage{amsfonts}
\usepackage{bbm}
\usepackage{amsmath}
\usepackage{amssymb}
\usepackage{epstopdf}
\usepackage{subfigure}
\usepackage{enumerate}
\RequirePackage{xspace}
\usepackage[ruled, section]{algorithm}
\usepackage{algpseudocode}
\RequirePackage[colorlinks,citecolor=blue,urlcolor=blue]{hyperref}

\renewcommand{\E}{\mathbb E}
\newcommand{ \R}{ \mathbb R}
\renewcommand{\Pr}{ \mathrm P}
\renewcommand{\D}{ \mathbb D}

\newcommand{ \cb}{ \mathcal B}

\newcommand{ \cV}{ \mathcal V}

\newcommand{\ga}{\alpha}
\newcommand{\gb}{\beta}
\newcommand{\gd}{\delta}
\newcommand{\gre}{\epsilon}
\newcommand{\gve}{\varepsilon}

\newcommand{\gl}{\lambda}
\newcommand{\gs}{\sigma}
\newcommand{\om}{\omega}

\newcommand{\gm}{\gamma}
\newcommand{\G}{\Gamma}

\newcommand{\bh}{\mathbf{H}}

\newcommand{\bw}{\mathbf{W}}
\newcommand{\bld}{\mathbf{D}}
\newcommand{\br}{\mathbf{R}}
\newcommand{\bc}{\mathbf{C}}

\newcommand{\rar}{ \rightarrow}

\newcommand{\V}[1]{\ensuremath{\boldsymbol{#1}}\xspace}

\makeindex             % used for the subject index
\begin{document}

\title*{Spectral Clustering and Block Models: A Review And A New Algorithm}
\author{Sharmodeep Bhattacharyya and Peter J. Bickel}
\institute{Sharmodeep Bhattacharyya \at Oregon State University, Department of Statistics, 44 Kidder Hall, Corvallis, OR, \email{bhattash@science.oregonstate.edu}
\and Peter J. Bickel \at University of California at Berkeley, Department of Statistics, 367 Evans Hall, Berkeley, CA \email{bickel@stat.berkeley.edu}}
\maketitle

\abstract{We focus on spectral clustering of unlabeled graphs and review some results on clustering methods which achieve weak or strong consistent identification in data generated by such models. We also present a new algorithm which appears to perform optimally both theoretically using asymptotic theory and empirically.}

\section{Introduction}
\label{intro}
%Use the template \emph{chapter.tex} together with the Springer document class SVMono (monograph-type books) or SVMult (edited books) to style the various elements of your chapter content in the Springer layout.
%
%Instead of simply listing headings of different levels we recommend to
%let every heading be followed by at least a short passage of text.
%Further on please use the \LaTeX\ automatism for all your
%cross-references and citations. And please note that the first line of
%text that follows a heading is not indented, whereas the first lines of
%all subsequent paragraphs are.

Since its introduction in \cite{MR0318007}, spectral analysis of various matrices associated to groups has become one of the most widely used clustering techniques in statistics and machine learning.

In the context of unlabeled graphs, a number of methods, all of which come under the broad heading of spectral clustering have been proposed. These methods based on spectral analysis of adjacency matrices or some derived matrix such as one of the Laplacians (\cite{shi2000normalized}, \cite{ng2002spectral}, \cite{MR2396807}, \cite{MR2893856}, \cite{MR3010899}) have been studied in connection with their effectiveness in identifying members of blocks in exchangeable graph block models. In this paper after introducing the methods and models, we intend to review some of the literature. We relate it to the results of Mossel, Neeman and Sly (2012) \cite{mossel2012stochastic} and Massouli\'{e} (2014) \cite{massoulie2014community}, where it is shown that for very sparse models, there exists a phase transition below which members cannot be identified better than chance and also showed that above the phase transition one can do better using rather subtle methods. In \cite{bhattacharyya2014community} we develop a spectral clustering method based on the matrix of geodesic distances between nodes which can achieve the goals of the work we cited and in fact behaves well for all unlabeled networks, sparse, semi-sparse and dense. We give a statement and sketch the proof of these claims in \cite{} but give a full argument for the sparse case considered by the above authors only in this paper. We give the necessary preliminaries in Section 2, more history in Section 3 and show the theoretical properties of the method in Section 4.

\section{Preliminaries}
\label{sec_prel}
There are  many standard methods of clustering based on numerical similarity matrices which are discussed  in a number of monographs (Eg:Hartigan \cite{MR0405726}, Leroy and Rousseuw \cite{MR914792}). We shall not discuss these further. Our focus is on unlabeled graphs of $n$ vertices characterized by adjacency matrices, $A=||a_{ij}||$ for $n$ data points. With $a_{ij} = 1$ if there is an edge between $i$ and $j$ and $a_{ij}=0$ otherwise. The natural assumption then is, $A=A^T$. Our basic goal is to divide the points in $K$ sets  such that on some average  criterion the points in a given subset are more similar to each other than to those of other subsets. Our focus is on methods of clustering based on the spectrum (eigenvalues and eigenvectors) of $A$ or related matrices. 
%In fact, we focus on the special case, $a_{ij}=0$ or $1$ ,where $A$ can be interpreted as the adjacency matrix of  an unlabeled  graph $G=(V,E)$ where $V$ is a set of $|V|=  n$ nodes  and $|E|$ edges where  an edge is  coded as an unordered pair of nodes.
 
%This is a latent variable model by which nodes are assigned to disjoint blocks  $1,\ldots,K$ independently with probabilities $\pi_1,\ldots,\pi_K$, $\sum_{j=1}^K\pi_j =1$ and given the memberships of the nodes, say $a$ and $b$, edges are assigned independently with probabilities, $P_{ab}$. We can now interpret correct clustering as correct ascription of nodes to blocks.

\subsection{Notation and Formal Definition of Stochastic Block Model}
\label{sec_sbm}
%As a model for community detection, we consider the stochastic block model. The stochastic block model is perhaps the most commonly used and best studied model for community detection. 
\begin{definition}
\label{def_sbm}
A graph $G^K(B, (P,\V{\pi}))$ generated from the \textbf{stochastic block model} (SBM) with $K$ blocks and parameters $P\in(0,1)^{K\times K}$ and $\V{\pi}\in (0,1)^K$ can be defined in following way - each vertex of graph $G_n$ is assigned to a community $\V{c} \in \{1, \ldots,K\}$. The $(c_1, \ldots, c_n)$ are independent outcomes of multinomial draws with parameter $\V{\pi} = (\pi_1, \ldots, \pi_K)$, where $\pi_i > 0$ for all $i$. Conditional on the label vector $\V{c} \equiv (c_1, \ldots, c_n)$, the edge variables $A_{ij}$ for $i < j$ are independent Bernoulli variables with
\begin{align}
\label{eq_sbm_adj}
\E[A_{ij}|\V{c}] = P_{c_ic_j} = \min\{\rho_nB_{c_ic_j}, 1\} ,
\end{align}
where $P = [P_{ab}]$ and $B = [B_{ab}]$ are $K \times K$ symmetric matrices. We call $P$ the \textbf{connection probability} matrix and $B$ the \textbf{kernel} matrix for the connection. So, we have $P_{ab} \leq 1$ for all $a, b = 1, \ldots, K$, $P\mathbf{1} \leq \mathbf{1}$ and $\mathbf{1}^TP\leq \mathbf{1}$ element-wise.
\end{definition}
By definition $A_{ji} = A_{ij}$, and $A_{ii} = 0$ (no self-loops). 

This formulation is a reparametrization due to Bickel and Chen (2009) \cite{bickel2009nonparametric} of the definition of Holland and Leinhardt \cite{holland1983stochastic}. It permits separate consideration asymptotically of the density of the graph and its structure as follows:
\begin{align*}
\Pr\left(\mbox{Vertex 1 belongs to block } a \mbox{ and vertex  2 to block }b\mbox{ and are connected}\right)=\pi_a\pi_b P_{ab}
\end{align*}
with $P_{ab}$  depending on n. $P_{ab}=\rho_n\min(B_{ab},1/\rho_n)$. We can interpret $\rho_n$  as the unconditional  probability of an edge and $B_{ab}$ essentially as  
\begin{align*}
\Pr\left(\mbox{Vertex 1 belongs to }a\mbox{ and vertex 2 belongs to }b | \mbox{ an edge between 1 and 2}\right).
\end{align*}

Set $\Pi = \mbox{diag}(\pi_1, \ldots, \pi_K)$. 
\begin{enumerate}
\item Define the matrices as $M = \Pi B$ and $S = \Pi^{1/2}B\Pi^{1/2}$.
\item Note that the eigenvalues of $M$ are the same as the symmetric matrix $S$ and in particular are real-valued. 
\item The eigenvalues of the expected adjacency matrix $\bar{A} \equiv \E(A)$ are also the same as those of $S$ but with multiplicities. We denote the eigenvalues by their absolute order, $\gl_1 \geq |\gl_2| \geq \cdots \geq |\gl_K|$. 
\end{enumerate}
Let us denote $(\varphi_1, \ldots, \varphi_K)$, $\varphi_i \in \R^K$, as the eigenvectors of $S$ corresponding to the eigenvalues $\gl_1, \ldots, \gl_K$. If a set of $\gl_j$'s are equal to $\gl$, we choose eigenvectors from the eigenspace corresponding to the $\gl$ as appropriate. Then, we have, $\phi_i = \Pi^{-1/2}\varphi_i$ and $\psi_i = \Pi^{1/2}\varphi_i$ as the left and right eigenvectors of $M$. Also, $\langle\phi_i, \phi_j\rangle_{\pi} = \sum_{k=1}^K \pi_k\phi_{ik}\phi_{jk}= \gd_{ij}$. The spectral decomposition of $M$, $S$ and $B$ are
\begin{align*}
B = \sum_{k=1}^K \gl_k\phi_k\phi_k^T,\hspace{0.2in} S = \sum_{k=1}^K \gl_k\varphi_k\varphi_k^T, \hspace{0.2in} M = \sum_{k=1}^K \gl_k\psi_k\phi_k^T.
\end{align*}

\subsection{Spectral Clustering}
\label{sec_specc}
The basic goal of community detection is to infer the node labels $\V{c}$ from the data. Although we do not explicitly consider parameter estimation, they can be recovered from $\hat{\V{c}}$, an estimate of $(c_1, \ldots, c_n)$ by
\begin{align}
\label{eq_meanlink}
\hat{P}_{ab}  \equiv  \frac{1}{O_{ab}}\sum_{i=1}^n\sum_{j=1}^n A_{ij}\mathbf{1}\left(\hat{\V{c}}_i = a, \hat{\V{c}}_j = b\right),\ \ \ 1\leq a,b\leq K,
\end{align}
where, 
\begin{eqnarray*}
%\label{eq_meanlink}
O_{ab}  \equiv & \left\{
	\begin{array}{ll}
	n_an_b, & \ \ 1\leq a,b\leq K, a\neq b \\
	n_a(n_a-1), & \ \ 1\leq a\leq K, a=b
\end{array}
\right. , &  n_{a}  \equiv  \sum_{i=1}^n \mathbf{1}\left(\hat{\V{c}}_i = a\right),\ 1\leq a\leq K 
\end{eqnarray*}

%We can also see SBM under a different generative process. We consider $\s$ to be a finite set, $(x_1, \ldots, x_n)\in [K]^n$ ($[K] = \{1, \ldots, K\}$) with $x_i \stackrel{iid}{\sim} Mult(n, \V{\pi})$ and kernel $\kappa:[K]\rar[K]$ as $\kappa(a,b) = B_{ab}$ ($a, b = 1, \ldots, K$), then the resulting generated graph follows stochastic block model. So, for SBM we can define an \textit{integral operator} on $[K]$ with measure $\{\pi_1, \ldots, \pi_K\}$. 
%\begin{definition}
%\label{def_op_sbm}
%The \textbf{integral operator} $T_B:\ell^1(\s) \rar \ell^1(\s)$ corresponding to \\ 
%$G^K(n, (P,\V{\pi}))$, is defined as 
%\begin{align*}
%(T_B(x))_a = \sum_{b=1}^K B_{ab}\pi_bx_b, \mbox{ for } a = 1, \ldots, K
%\end{align*}
%where, $x\in \R^K$.
%\end{definition}

%The stochastic block model has deep connections with Multi-type branching process, just as, Erod\"{o}s-R\'{e}nyi random graph model (ERRGM) has connections with the branching process. The connections will dealt in more detail in Section \ref{sec_geo_theory}.

%\subsection{Algorithms for Community Detection}
%\label{sec_algo_review}

%Adjacency matrix, $A$ by now is intensely studied \cite{MR3010899} based on probabilistic generating model for graphs as the $K$ stochastic block model \cite{holland1983stochastic}. 
There are a number of approaches for community detection based on modularities (\cite{girvan2002community}, \cite{bickel2009nonparametric}), maximum likelihood and variational likelihood (\cite{MR2988467}, \cite{MR3127853}) and approximations such as  semidefinite programming approaches \cite{amini2014semidefinite}, pseudolikelihood \cite{MR3127859} but these all tend to be computationally intensive  and/or require good initial assignments of blocks. The methods which have proved both computationally effective and asymptotically correct in a sense we shall discuss are related to spectral analysis of the adjacency or related matrices.They differ in important details.  

Given an $n\times n$  symmetric matrix $M$ based on $A$, the algorithms are of the form:
\begin{enumerate}
\item Using the spectral decomposition of $M$ or a related generalized eigenproblem.
\item Obtain an  $n\times K$ matrix of $K$ $n\times 1$ vectors. 
\item Apply $K$ means clustering to  the $n$  $K$-dimensional  row  vectors  of the matrix  of Step 2.
\item Identify  the indices of the rows  belonging to cluster $j$ ,$j=1,\ldots,K$  with vertices belonging to block $j$.
\end{enumerate}

In addition to $A$, three graph Laplacian matrices discussed by von Luxburg (2007) \cite{von2007tutorial}, have been considered  extensively, as well as some others we shall mention briefly below and the matrix we shall  show has optimal asymptotic properties and discuss in greater detail. The matrices popularly considered are:
\begin{itemize}
\item $L=D-A$:  the graph Laplacian.
\item $L_{\mbox{rw}} = D^{-1}A$: the random walk Laplacian.
\item $L_{\mbox{sym}} =D^{-1/2}AD^{-1/2}$: the symmetric Laplacian. 
\end{itemize}

Here $D =\mbox{diag}(A\mathbf{1})$, the diagonal matrix whose diagonal is  the vector  of row sums of $A$.
She considers  optimization problems which are relaxed versions of  combinatorial problems which implicitly define clusters as sets of nodes with more internal than external edges. $L$ and $L_{\mbox{sym}}$ appear in two of these relaxations. 
%$L_{\mbox{sym}}$ is naturally related to algorithms such as Page Rank (Ref).

The form of step 2 differs for $L$ and $L{\mbox{sym}}$  with  the $K$ vectors of the $L$ problem corresponding to the top $K$ eigenvalues of the generalized eigenvalue problem $Lv=\gl Dv$ ,while the $n$ $K$-dimensional vectors of the $L_{\mbox{sym}}$ problem are obtained by normalizing the rows of the matrix of $K$ eigenvectors corresponding to the  top $K$ eigenvalues  of $L_{\mbox{sym}}$. Their relation to the $K$ block model is through asymptotics. 

Why is spectral clustering expected to work? Given $A$ generated by a $K$-block model, let $\V{c} \leftrightarrow (n_1, \ldots, n_K)$ where, $n_a$ is the number of vertices assigned to type $a$. Then we can write,
\begin{align*}
\E(A|\V{c}) = PQP^T
\end{align*}
where, $P$ is a permutation matrix and $Q_{n\times n}$ has succesive blocks of $n_1$ rows, $n_2$ rows and so on with all the vectors in each row the same. Thus $\mbox{rank}(\E(A|\V{c}) = K$. The same is true of the asymptotic limit of $L$ given $\V{c}$.

If asymptotics as $n\rar \infty$ justify concentration of $A$ or $L$ around their expectations then we expect all eigenvalues other than the largest $K$ in absolute value are small. It follows that the $n$ rows of the $K$ eigenvectors associated with the top $K$ eigenvalues should be resolvable into $K$ clusters in $\R^K$ with cluster members identified with rows of $A_{n\times n}$, see \cite{MR2893856}, \cite{MR3010899} for proofs.

\subsection{Asymptotics}
\label{sec_asymp}
Now we can consider several  asymptotic regimes  as $n\rar \infty$. Let $\gl_n=n\rho_n$  be the average degree of the graph.
\begin{itemize}
\item[(I)] \hspace{0.05in} The \emph{dense} regime:  $\lambda_n = \Omega(n)$.
\item[(II)] \hspace{0.05in} The \emph{semi dense} regime: $\lambda_n/log(n) \rar \infty$.
\item[(III)]  \hspace{0.05in} The \emph{semi sparse} regime: Not semidense but $\lambda_n\rar \infty$.
\item[(IV)] \hspace{0.05in} The \emph{sparse} regime: $\lambda_n = O(1)$.
\end{itemize}

Here are some results in the different regimes. We define a method of vertex assignment  to communities  as a random map $\delta:\{1,\ldots,n\} \rar \{1, \ldots,K\}$ where randomness comes through the dependence of delta on $A$ as a function. Thus spectral clustering using  the various matrices  which depend on $A$ is such a $\delta$.

\begin{definition}
$\delta$ is said to be \emph{strongly consistent} if 
\begin{align*}
\Pr(i\mbox{ belongs to }a\mbox{ and }\delta(i)=a\mbox{ for all }i , a)\rar 1\mbox{ as } n\rar \infty.
\end{align*}
\end{definition}
Note that  the blocks are only determined up to permutation.

Bickel and Chen (2009) \cite{bickel2009nonparametric} show that  in the (semi) dense regime a method called profile likelihood is strongly consistent  under minimal identifiability conditions and later this result was extended \cite{MR3127853} to fitting by maximum likelihood or variational likelihood. In fact, in the (semi) dense regime, the block model likelihood asymptotically agrees with the joint likelihood of  $A$ and vertex block identities so that efficient estimation of all parameters is possible. It is easy to see that the result cannot  hold in the (semi)sparse regime since  isolated points then exist with probability 1.

Unfortunately all of these methods are computationally  intensive. Although spectral clustering is not strongly consistent, a slight variant, reassigning vertices in any cluster $a$ which are maximally connected to another cluster $b$ rather than $a$ , is strongly consistent.
\begin{definition}
$\delta$ is said to be \emph{weakly consistent} if and only if 
\begin{align*}
W \equiv n^{-1}\sum_{i=1}^n\Pr\left(i \in a,\delta(i)\neq a | \forall i,a\right)=o(1)
\end{align*}
\end{definition}

Spectral clustering applied  to  $A$ \cite{MR3010899} or the Laplacians (\cite{MR2893856} in the manner we have described) has been shown to be weakly consistent in the semi dense to dense regimes. Even weak consistency fails for parts of the sparse regime \cite{abbe2014exact}. The best that can be hoped for is $W < \frac{1}{2}$. A sharp problem has been posed and eventually resolved  in a series of papers, Decelle et al \cite{decelle2011asymptotic}, Mossel et al \cite{mossel2013proof}. These writers considered the  case $K=2, \pi_1=\pi_2, B_{11}=B_{22}$. First, Decelle  et al. \cite{decelle2011asymptotic} argued on physical grounds that if, $F=2(B_{11}-B_{12})^2/(B_{11}+B_{12}) \leq 1$, then  $W \geq 1/2$  for any method  and parameters are unestimable from the data even if they satisfy the minimal identifiability conditions  given below. On the other hand  Mossel et al \cite{mossel2013proof} and independently  Massoulie et al \cite{massoulie2014community}, devised admittedly slow methods such that if $F>1$ then $W<1/2$ and parameters can be estimated consistently. 

We now present a fast spectral clustering method given in greater detail in \cite{bhattacharyya2014community} which yields weak consistency  for the semisparse regime on and also has the properties of  the Mossel et al and Massoulie methods. In fact, it reaches the phase transition threshold  for all K not just K=2, but still restricted to $\pi_j=1/K$, all $j$ and $B_{aa}+2\sum \left[B_{ab}: b \neq a\right]$  independent of $a$ for all $a$. 

We note that Zhao et. al. (2015) \cite{gao2015achieving} exhibit a two-stage algorithm which exhibits the same behavior but its properties in sparse case are unknown. The algorithm given in the next section involves spectral clustering of a new matrix, that of all geodesic distances between $i$ and $j$. 
%It is presented in greater detail in \cite{}.

\section{Algorithm}
\label{sec_prel}
% Always give a unique label
% and use \ref{<label>} for cross-references
% and \cite{<label>} for bibliographic references
% use \sectionmark{}
% to alter or adjust the section heading in the running head
%Let us suppose that we have a random graph $G_n$ as the data. Let $V(G_n) = \{v_i, \ldots, v_n\}$ denote the vertices of $G_n$ and $E(G_n) = \{e_1, \ldots, e_m\}$ denote the edges of $G_n$. So, the number of vertices in $G_n$ is $|V(G_n)| = n$ and number of edges of $G_n$ is $|E(G_n)| = m$. Let the adjacency matrix of $G_n$ be denoted by $A_{n\times n}$. For the sake of notational simplicity, from here onwards we shall denote $G_n$ by $G$ having $n$ vertices unless specifically mentioned. We consider the $n$ vertices of $G$ are clustered into $K$ different communities with each community having size $n_a$, $a=1,\dots, K$ and $\sum_a n_a = n$. In this paper, we are interested in the problem of \textit{vertex community identification} or \textit{graph partitioning}. That means that we are interested in finding which of the $K$ different community each vertex of $G$ belongs to. However, the problem is an \textit{unsupervised learning} problem. So, we assume that the data is coming from an underlying model and we try to verify how good `our' \textit{community detection} method works for that model.
As usual let $G_n$, an undirected graph on $n$ vertices be the data. denote the vertex set by $V(G_n) \equiv \{v_1, \ldots, v_n\}$ and the edge set by $E(G_n) \equiv \{e_1, \ldots, e_m\}$ with cardinalities $|V(G_n)| = n$ and $E(G_n)| = m$.

%Thus we are not really interested in estimation or inference on parameters $\V{\pi}$ and $P$, but, rather we are interested in estimating $\V{c}$. But, it does not mean the two problems are mutually exclusive, in reality, the inferential problem and the community detection problem are quite interlinked. 

%The stochastic block model can be considered as special case of some general nonparametric network models. Those models are briefly described below as we might use results on the general nonparametric models in stead of just stochastic block models.

%\subsection{Algorithm}
%\label{sec_algo}

As usual a path between vertices $u$ and $v$ is a set of edges $\{(u, v_1), (v_1, v_2), \ldots, (v_{\ell-1}, v)\}$ and the length of such a path is $\ell$.

The algorithm we propose depends on the graph distance or geodesic distance between vertices in a graph.
\begin{definition}
\label{def_geo_dis}
The \textbf{Graph} or \textbf{Geodesic distance} between two vertices $i$ and $j$ of graph $G$ is given by the length of the shortest path between the vertices $i$ and $j$, if they are connected. Otherwise, the distance is infinite.
\end{definition}
So, for any two vertices $u, v\in V(G)$, graph distance, $d_g$ is defined by
\begin{eqnarray*}
\label{eq_geo_dis}
d_g(u, v) &=&\left\{
\begin{array}{l}
    \min\{\ell |\exists \mbox{ path of length } \ell \mbox{ between } u \mbox { and } v\},   \\
    \infty, \mbox{ if } u \mbox{ and } v \mbox{ are not connected} 
     \end{array}
     \right.
\end{eqnarray*}
For implementation, we can replace $\infty$ by $n+1$, when, $u$ and $v$ are not connected, since any path with loops can not be a geodesic. The main steps of the algorithm are as follows
\begin{enumerate}
\item[1.] Find the graph distance matrix $D = [d_g(v_i, v_j)]_{i, j=1}^n$ for a given network but with distance upper bounded by $k\log n$. Assign non-connected vertices an arbitrary high value.
\item[2.] Perform hierarchical clustering to identify the giant component $G^C$ of graph $G$. Let $n_C = |V(G^C)|$.
\item[3.] Normalize the graph distance matrix on $G^C$, $D^C$ by 
\begin{align*}
\bar{D}^C = -\left(I-\frac{1}{n_C}\mathbf{1}\mathbf{1}^T\right)(D^C)^2\left(I-\frac{1}{n_C}\mathbf{1}\mathbf{1}^T\right) 
\end{align*}
\item[4.] Perform eigenvalue decomposition on $\bar{D}^C$.
\item[5.] Consider the top $K$ eigenvectors of normalized distance matrix $\bar{D}^C$ and $\tilde{\bw}$ be the $n\times K$ matrix formed by arranging the $K$ eigenvectors as columns in $\tilde{\bw}$. Perform $K$-means clustering on the rows $\tilde{\bw}$, that means, find an $n\times K$ matrix $\bc$, which has $K$ distinct rows and minimizes $||\bc - \tilde{\bw}||_F$.
\item[6.] (Alternative to 5.) Perform Gaussian mixture model based clustering on the rows of $\tilde{\bw}$, when there is an indication of highly-varying average degree between the communities.
\item[7.] Let $\hat{\V{c}} : V \mapsto [K]$ be the block assignment function according to the clustering of the rows of $\tilde{\bw}$ performed in either Step 5 or 6.
\end{enumerate}
Here are some important observations about the implementation of the algorithm -
\begin{itemize}
\item[(a)] \hspace{0.02in} There are standard algorithms for graph distance finding in the algorithmic graph theory literature. In the algorithmic graph theory literature the problem is known as the \textbf{all pairs shortest path} problem. The two most popular algorithms are Floyd-Warshall \cite{floyd1962algorithm} \cite{warshall1962theorem} and Johnson's algorithm \cite{johnson1977efficient}. 
\item[(b)] \hspace{0.02in} Step 3 of the algorithm is nothing but the classical multi-dimensional scaling (MDS) of the graph distance matrix. 
\item[(c)] \hspace{0.02in} In the Step 5 of the algorithm $K$-means clustering is appropriate if the expected degree of the blocks are equal. However, if the expected degree of the blocks are different, this leads to multi scale behavior in the eigenvectors of the normalized distance matrix and bad behavior in practice. So, we perform Gaussian Mixture Model (GMM) based clustering instead of $K$-means to take into account that.
\end{itemize}

General theoretical results on the algorithm will be given in \cite{bhattacharyya2014community}. In this paper, we first restrict to the sparse regime 
%and a simple but widely studied case whose generality can be be derived \cite{gao2015achieving}.  
We do so because the arguments in the sparse regime are essentially different from the others. Curiously, it is in the sparse and part of the semi-sparse regime only that the matrix $\bar{D}^C$ concentrates to an $n \times n$ matrix with $K$ distinct types of row vectors as for the other methods of spectral clustering. It does not concentrate in the dense regime, while the opposite is true of $A$ and $L$. They do not concentrate outside the semidense regime. That the geodesic matrix does not concentrate in the dense regime can easily be seen since asymptotically all geodesic paths are of constant length. But the distributions of path lengths differs from block to block ensuring that the spectral clustering works. But we do not touch this further here.

%Now, the main theoretical result on the method will be proved in the next section. However, we shall write down the statement of the theorem here. The assumptions under which the Theorem holds is stated in Section \ref{sec_theo_res}.
%\begin{theorem}
%\label{thm_mis}
%Under the conditions (A1)-(A3) in Section \ref{sec_theo_res}, suppose that the number of blocks $K$ is known. Let $\hat{\xi} : V \mapsto [K]$ be the block assignment function according to a clustering of the rows of $\bar{\bw}^{(n)}$ satisfying algorithm and $\V{c} : V \mapsto [K]$ be the actual assignment. Let $\mathcal{P}_K$ be the set of permutations on $[K]$. With high probability and for large $n$ it holds that for some $\gve > 0$ and when the size of giant component of $G_n$ is $gn$, 
%\begin{align}
%\label{eq_miscl}
%\min_{\pi\in\mathcal{P}_K} |\{u \in V: \xi(u)\neq \pi(\hat{\xi}(u))\}| = O((gn)^{1-\gve})
%\end{align}
%\end{theorem}

\section{Theoretical Results}
\label{sec_geo_theory}
%Let the adjacency matrix of $G_n$ be denoted by $A_{n\times n}$. For sake of notational simplicity, from here onwards we shall denote $G_n$ by $G$ having $n$ vertices unless specifically mentioned. There are $K$ communities for the vertices and each community has $(n_a)_{a=1}^K$ number of vertices. In this paper, we are interested in the problem of \textit{vertex community identification} or \textit{graph partitioning}. However, the problem is an \textit{unsupervised learning} problem. So, we assume that the data is coming from an underlying model and we try to verify how good `our' \textit{community detection} method works for that model.

Throughout this section we take $\rho_n = \frac{1}{n}$ and specialize to the case 
\begin{align*}
B = (p - q)\mbox{I}_{K\times K} + q\mathbf{1}\mathbf{1}^T
\end{align*}
where, $\mbox{I}$ is the identity and $\mathbf{1} = (1, \ldots, 1)^T$. That is, all $K$ blocks have the same probability $p$ of connecting two block members and probability $q$ of connecting members of two different blocks and $p > q$. We also assume that $\pi_a = \frac{1}{K}$, $a = 1, \ldots, K$, all blocks are asymptotically of the same size. 
%As \cite{gao2015achieving} show one can pass from this type of model to a general case models with some weak assumptions. 
We restrict ourselves to this model here because it is the one treated by Mossel, Neeman and Sly (2013) \cite{mossel2013proof} and already subtle technical details are not obscured. Here is the result we prove.
\begin{theorem}
\label{thm_mis}
For the given model, if
\begin{align}
\label{eq_main_cond}
(p-q)^2 > K(p+(K-1)q),
\end{align}
and our algorithm is applied, $\hat{\V{c}}$ results and $\V{c}$ is the true assignment function, then, 
\begin{align}
\label{eq_miscl}
\left[\frac{1}{n}\sum_{i=1}^n\mathbf{1}\left(\V{c}(v_i) \neq \hat{\V{c}}(v_i)\right) < \frac{1}{2}\right] \rar 1
\end{align}
\end{theorem}
\textbf{Notes:} \begin{enumerate}
\item \eqref{eq_main_cond} marks the phase transition conjectured by \cite{decelle2011asymptotic}.
\item A close reading of our proof shows that as $(p-q)^2/K(p+(K-1)q) \rar \infty$, $\frac{1}{n}\sum_{i=1}^n\mathbf{1}\left(\V{c}(v_i) \neq \hat{\V{c}}(v_i)\right) \stackrel{P}{\rar} 0$.
%\item It is possible using our technique to derive results such as \eqref{eq_miscl} for more general $B$. We do not presume this but indicate at various steps of the argument where modifications are needed.
\end{enumerate}

We conjecture that our conclusion in fact holds under the following conditions,
\begin{itemize}
\item[(A1)]\label{ass_main1}\hspace{0.1in} We consider $\gl_1 > 1$, $\gl_1> \max_{j \geq 2} \gl_j$, $1\leq j\leq K$ and $\gl_K > 0$. For $M$, there exists a $k$ such that $(M^k)_{ab} > 0$ for all $a,b = 1, \ldots, K$. Also, $\pi_j > 0$, for $j=1, \ldots, K$.
%It implies that $\gl_1 > \max_{k\geq 2} |\gl_k|$ and $\pi_k > 0$, for all $k\in[K]$. Also, define $K_0$ by $\gl_k^2 > \gl_1$ for all $k\in[K_0]$ and $\gl^2_{K_0+1} \leq \gl_1$.
\item[(A2)]\label{ass_main2}\hspace{0.1in} Each vertex has the same asymptotic average degree $\ga > 1$, that is,
\begin{align*}
\ga = \sum_{k=1}^K \pi_kB_{ak} = \sum_{k=1}^K M_{ak}, \hspace{0.1in} \mbox{ for all } a\in \{1,\ldots,K\} 
\end{align*}
%\item[(A3)]\hspace{0.1in} We also assume that for some $\gm \in (0,1]$,
%\begin{align*}
%||\pi - \pi_n||_{\infty} = \max_{k\in [K]} |\pi_k - \pi_n(k)| = O(n^{-\gm}) 
%\end{align*}
\item[(A3)]\label{ass_main3}\hspace{0.1in} We assume that
\begin{align*}
\gl_K^2 > \gl_1
\end{align*} 
or alternatively, there exists real positive $t$, such that,
\begin{align*}
\sum_{k=1}^K \phi_k(a)\gl_k^t\phi_k(b) \leq n,\ \ \ \mbox{ for all } a,b =1, \ldots, K
\end{align*}

%\item[(A3)]\label{ass_main3}\hspace{0.1in} We also consider another assumption, which is not necessary but helps simplify the calculations in the paper. We consider that $B = (p-q)\mbox{I}_{K\times K} + q\mathbf{1}\mathbf{1}^T$ and $\V{\pi} = (1/K, \ldots, 1/K)$. So, $\ga = \gl_1 = (p+(K-1)q)/K$ and $\gl_2 = \ldots = \gl_K = (p-q)/K$. Along with assumption (A1), we consider that, $\gl_2^2 > \gl_1$ or
%\begin{align*}
%(p-q)^2 \geq K\left(p+(K-1)q\right).
%\end{align*}
%Also, $\ga = \gl_1 = O(1)$.
\end{itemize}
Note that (A1)-(A3) all hold for the case we consider. In fact, under our model,
\begin{align*}
\gl_1 = \frac{p + (K-1)q}{K},\ \ \ \gl_2 = \frac{p-q}{K},\ \ \ \gl_2 = \gl_3 = \cdots = \gl_K
\end{align*}
with (A3) being the condition of the Theorem.

Our argument will be stated in a form that is generalizable and we will indicate revisions in intermediate statements as needed, pointing in particular to a lemma whose conclusion only holds if an implication of (A3) we conjecture is valid.

The theoretical analysis of the algorithm has two main parts -
\begin{itemize}
\item[I.] Finding the limiting distribution of graph distance between two typical vertices of type $a$ and type $b$ (where, $a, b = 1, \ldots, K$). This part of the analysis is highly dependent on results from multi-type branching processes and their relation with stochastic block models. The proof techniques and results are borrowed from \cite{bollobas2007phase}, \cite{bhamidi2011first} and \cite{athreya1972branching}.
\item[II.] Finding the behavior of the top $K$ eigenvectors of the graph distance matrix $D$ using the limiting distribution of the typical graph distances. This part of analysis is highly dependent on perturbation theory of linear operators. The proof techniques and results are borrowed from \cite{kato1995perturbation}, \cite{chatelin1983spectral} and \cite{MR3010899}.
\end{itemize}
%This proof rests on the following branching process results.
We will state two theorems corresponding to I and II above.

%\subsection{Results of Part I}
%\label{sec_theo_res}
%Recall the definitions of $B$, $M$ and $S$ from Section \ref{sec_sbm}. We consider the following assumptions -
%\begin{itemize}
%\item[(A1)]\label{ass_main1}\hspace{0.1in} We consider $\gl_1 > 1$, $\gl_1 = O(1)$ and $M$ is positively regular, there exists a $k$ such that $(M^k)_{ab} > 0$ for all $a,b = 1, \ldots, K$.
%%It implies that $\gl_1 > \max_{k\geq 2} |\gl_k|$ and $\pi_k > 0$, for all $k\in[K]$. Also, define $K_0$ by $\gl_k^2 > \gl_1$ for all $k\in[K_0]$ and $\gl^2_{K_0+1} \leq \gl_1$.
%\item[(A2)]\label{ass_main2}\hspace{0.1in} Each vertex has the same asymptotic average degree $\ga > 1$, that is,
%\begin{align*}
%\ga = \sum_{k=1}^K \pi_kB_{ak} = \sum_{k=1}^K M_{ak}, \hspace{0.1in} \mbox{ for all } a\in \{1,\ldots,K\} 
%\end{align*}
%%\item[(A3)]\hspace{0.1in} We also assume that for some $\gm \in (0,1]$,
%%\begin{align*}
%%||\pi - \pi_n||_{\infty} = \max_{k\in [K]} |\pi_k - \pi_n(k)| = O(n^{-\gm}) 
%%\end{align*}
%\item[(A3)]\label{ass_main3}\hspace{0.1in} We also consider another assumption, which is not necessary but helps simplify the calculations in the paper. We consider that $B = (p-q)\mbox{I}_{K\times K} + q\mathbf{1}\mathbf{1}^T$ and $\V{\pi} = (1/K, \ldots, 1/K)$. So, $\ga = \gl_1 = (p+(K-1)q)/K$ and $\gl_2 = \ldots = \gl_K = (p-q)/K$. Along with assumption (A1), we consider that, $\gl_2^2 > \gl_1$ or
%\begin{align*}
%(p-q)^2 \geq K\left(p+(K-1)q\right).
%\end{align*}
%Also, $\ga = \gl_1 = O(1)$.
%\end{itemize}
\begin{theorem}
\label{thm_geo_dis_bnd}
Under our model, the graph distance $d_G(u, v)$ between two uniformly chosen vertices of type $a$ and $b$ respectively, conditioned on being connected, satisfies the following asymptotic relation -
\begin{itemize}
%\item[(i)] If $a \neq b$,
%\begin{align}
%\label{eq_geo_diff1}
%\Pr\left[d_G(u,v) \leq (1+\gve)\frac{\log n}{\log\gl}\right] = 1-o(1)
%\end{align}
%\begin{align}
%\label{eq_geo_diff1}
%\Pr\left(\sqrt{n}\left|\frac{\log |\gl|}{\log(n/(v_av_b))}d_G(u, v) - 1\right| > \gre\right) = o(1)
%\end{align}
%\begin{align}
%\label{eq_geo_diff1}
%\Pr\left[d_G(u, v) < (1-\gve)\frac{\log n}{\log |\gl|/\log(\nu_a\nu_b)}\right] = o(1)
%\end{align}
\item[(i)] If $a = b$, for any $\gve > 0$, as $n \rar \infty$, 
% and if $K_{aa}\pi_a > (K_{aa})^{\gm_a}\ \forall a$ with $\gm = \min_a \gm_a$,
\begin{align}
\label{eq_geo_diff1a}
\Pr\left[(1-\gve)\tau_1\leq d_G(u,v) \leq (1+ \gve)\tau_1\right] = 1 - o(1)
\end{align}
%\begin{align}
%\label{eq_geo_diff1b}
%\Pr\left[d_G(u,v) \geq (1+ \gve)\frac{\log n\pi_a}{2\log (\gl)}\right] = o(1)
%\end{align}
where, $\tau_1$ is the minimum real positive $t$, which satisfies the relation below,
\begin{align}
\label{eq_sbm_ev1}
\left[\gl_2^t +\frac{\gl_1^t - \gl_2^t}{K}\right] = n
\end{align}
\item[(ii)] 
%\begin{align}
%\label{eq_geo_diff2}
%\Pr\left(\sqrt{n}\left|\frac{\log |\gl|}{\log(n/(v_av_b))}d_G(u, v) - 1\right| > \gre\right) = o(1)
%\end{align}
%\begin{align}
%\label{eq_geo_diff2}
%\Pr\left[d_G(u, v) > (1+\gve)\frac{\log n}{\log |\gl|/\log(\nu_a\nu_b)}\right] = o(1)
%\end{align}
If $a\neq b$, for any $\gve > 0$, as $n \rar \infty$, 
\begin{align}
\label{eq_geo_diff2a}
\Pr\left[(1-\gve)\tau_2\leq d_G(u,v) \leq (1+ \gve)\tau_2\right] = 1 - o(1)
\end{align}
%\begin{align}
%\label{eq_geo_diff2b}
%\Pr\left[d_G(u,v) \geq (1+ \gve)\frac{\log n}{\log \gl}\right] = o(1)
%\end{align}
where, $\tau_2$ is the minimum real positive $t$, which satisfies the relation below, 
\begin{align}
\label{eq_sbm_ev2}
\left[\frac{\gl_1^t - \gl_2^t}{K}\right] = n
\end{align}
\end{itemize}
\end{theorem}
%\textbf{Remark:} Note that, the equations \eqref{eq_sbm_ev1} and \eqref{eq_sbm_ev2} can be modified for general $B$ as $K(K+1)/2$ equations of the form
%\begin{align}
%\label{eq_sbm_ev}
%\sum_{k=1}^K \phi_k(a)\gl_k^t\phi_k(b) = (S^t)_{ab} = n, \ \ \ \mbox{ for } a\leq b \in [K]
%\end{align}
%The proof of the modified theorem general $B$ with equations \eqref{eq_sbm_ev} will follow the same line as in the proof of Theorem \ref{thm_geo_dis_bnd}.

In Theorem \ref{thm_geo_dis_bnd} we have a point-wise result. To use matrix perturbation theory for part II we need the following.
%so, we combine these point-wise results to give a matrix result -
\begin{theorem}
\label{thm_geo_dis}
%Under the assumptions (A1)-(A3), within the big connected component of graph $G_n$,
Let $\D_B$ be the restriction of the geodesic matrix to vertices in the big component of $G_n$. Then, under our model, 
\begin{align*}
\Pr\left[\left|\left|\frac{\bld}{\log n} - \D\right|\right|_F \leq o(n)\right] = 1-o(1)
\end{align*}
where, $\D_{ij} \equiv \gs_1 = \tau_1/\log n$, if $v_i$ and $v_j$ have same type and $\D_{ij} \equiv \gs_2 = \tau_2/\log n$, otherwise, where, $\tau_1$ and $\tau_2$ are solutions $t$ in Eq. \eqref{eq_sbm_ev1} and \eqref{eq_sbm_ev2} respectively.
\end{theorem}
To generalize Theorem \ref{thm_mis}, we need appropriate generalizations of Theorem \ref{thm_geo_dis_bnd} and \ref{thm_geo_dis}. Heuristically, it may be argued that the generalizations $(\tau_{sb})$, $a,b=1, \ldots, K$ should satisfy the equations,
\begin{align}
\label{eq_sbm_ev}
\sum_{k=1}^K \phi_k(a)\gl_k^t\phi_k(b) = (S^t)_{ab} = n, \ \ \ \mbox{ for } a\leq b \in [K]
\end{align}
Our conjecture is that (A1)-(A3) imply that the equations have asymptotic solutions and that the statements of Theorem \ref{thm_geo_dis_bnd} and \ref{thm_geo_dis} hold with obvious modifications.

Note that in Theorem \ref{thm_geo_dis_bnd}, since $\gl_j = \gl_2$, $2\leq j\leq K$ there are effectively only two equations and modifications are also needed for other degeneracies in the parameters. We next turn to a branching process result in \cite{bordenave2015non} which we will use heavily.

\subsection{A Key Branching Process Result}
\label{sec_br_proc}
As others have done we link the network formed by SBM with the tree network generated by multi-type Galton-Watson branching process. In our case, the Multi-type branching process (MTBP) has type space $S = \{1, \ldots, K\}$, where a particle of type $a \in S$ is replaced in the next generation by a set of particles distributed as a Poisson process on $S$ with intensity $(B_{ab}\pi_b)_{b = 1}^K = (M_{ab})_{b=1}^K$. Recall the definitions of $B$, $M$ and $S$ from Section \ref{sec_sbm}. We denote this branching process, started with a single particle of type $a$, by $\cb_{B,\pi}(a)$. We write $\cb_{B,\pi}$ for the same process with the type of the initial particle random, distributed according to $\V{\pi}$. According to Theorem 8.1 of Chapter 1 of \cite{mode1971multitype}, the branching process has a positive survival probability if $\gl_1 > 1$, where, $\gl_1$ is the Perron-Frobenius eigenvalue of $M$, a positive regular matrix. Recall that for our special $M$, $\gl_1 = \frac{p - q}{K} + 1$.

\begin{definition}
\label{def_survival_prob}
\begin{itemize}
%\item[(a)] Define $\rho_k (B,\pi; a)$ as the probability that the branching process \\ 
%$\cb_{B,\pi}(a)$ has a total population of exactly $k$ particles.
%\item[(b)] Define $\rho_{\geq k}(B,\pi; a)$ as the probability that the total population is at least $k$. 
\item[(a)] Define $\rho(B,\pi; a)$ as the probability that the branching process, $\cb_{B,\pi}(a)$, survives for eternity. 
\item[(b)] Define,
%\begin{align}
%\label{eq_br_surv}
%\rho_k(B,\pi) \equiv \sum_{a=1}^K\rho_k(B,\pi; a)\pi_a, \ \ \ \rho\equiv \rho(B,\pi) \equiv \sum_{a=1}^K \rho(B,\pi; a)\pi_a
%\end{align}
\begin{align}
\label{eq_br_surv}
\rho\equiv \rho(B,\pi) \equiv \sum_{a=1}^K \rho(B,\pi; a)\pi_a
\end{align}
as the \textbf{survival probability} of the branching process $\cb_{B,\pi}$ given that its initial distribution is $\V{\pi}$
\end{itemize}
\end{definition}

We denote $Z_t = (Z_t(a))_{a=1}^K$ as the population of particles of $K$ different types, with $Z_t(a)$ denoting particles of type $a$, at generation $t$ for the Poisson multi-type branching process $\cb_{B,\pi}$, with $B$ and $\V{\pi}$ as defined in Section \ref{sec_geo_theory}. From Theorem 24 of \cite{bordenave2015non}, we get that 
\begin{theorem}[\cite{bordenave2015non}]
\label{thm_bp1}
Let $\gb > 0$ and $Z_0 = x \in \mathbb{N}^K$ be fixed. There exists $C=C(x,\gb) > 0$ such that with probability at least $1 - n^{-\gb}$, for all $k\in [K]$, all $s,t \geq 0$, with $0\leq s< t$, 
\begin{align}
\label{eq_bp1}
|\langle \phi_k, Z_s\rangle - \gl_k^{s-t}\langle \phi_k, Z_t\rangle| \leq C(t+1)^2 \gl_1^{s/2}(\log n)^{3/2}
\end{align}
%and for all $k \in [K]\backslash [K_0]$, for all $t \geq 0$, 
%\begin{align}
%\label{eq_bp2}
%|\langle \phi_k, Z_t\rangle| \leq C (t+1)^2\gl_1^{t/2}(\log n)^{3/2}
%\end{align}
%Finally, for all $k \in [K]\backslash [K_0]$, all $t \geq 0$, $\E|\langle \phi_k, Z_t\rangle|^2 \leq C(t+1)^3\gl_1^t$.
\end{theorem}
\begin{itemize}
\item[\textbf{Remark:}] \hspace{0.4in}The above stated theorem is a special case of the general theorem stated in \cite{bordenave2015non}. The general theorem is required for generalizing Theorem \ref{thm_mis}. The general version of the theorem is
\begin{theorem}[\cite{bordenave2015non}]
\label{thm_bp2}
Let $\gb > 0$ and $Z_0 = x \in \mathbb{N}^K$ be fixed. There exists $C=C(x,\gb) > 0$ such that with probability at least $1 - n^{-\gb}$, for all $k\in [K_0]$ (where, $K_0$ is the largest integer such that $\gl_k^2 > \gl_1$ for all $k \leq K_0$), all $s,t \geq 0$, with $0\leq s< t$, 
\begin{align}
\label{eq_bp1}
|\langle \phi_k, Z_s\rangle - \gl_k^{s-t}\langle \phi_k, Z_t\rangle| \leq C(t+1)^2 \gl_1^{s/2}(\log n)^{3/2}
\end{align}
and for all $k \in [K]\backslash [K_0]$, for all $t \geq 0$, 
\begin{align}
\label{eq_bp2}
|\langle \phi_k, Z_t\rangle| \leq C (t+1)^2\gl_1^{t/2}(\log n)^{3/2}
\end{align}
Finally, for all $k \in [K]\backslash [K_0]$, all $t \geq 0$, $\E|\langle \phi_k, Z_t\rangle|^2 \leq C(t+1)^3\gl_1^t$.
\end{theorem}
\end{itemize}

\subsection{The Neighborhood Exploration Process}
\label{sec_nbhd_exp}
The neighborhood exploration process of a vertex $v$ in graph $G$ generated from an SBM gives us a handle on the link between local structures of a graph from SBM and multi-type branching process. Recall the definitions of SBM parameters from Section \ref{sec_sbm} and the definitions of Poisson multi-type branching process from Section \ref{sec_br_proc} . We assume all vertices of graph $G_n$ generated from a stochastic block model has been assigned a community or type $\xi_i$ (say) for vertex $v_i \in V(G_n)$. 
%\begin{definition}
%\label{def_nbd_expl_proc}

The \textit{neighborhood exploration process}, $(G, v)_L$, of a vertex $v$ in graph $G_n$, generates a \emph{spanning tree} of the induced subgraph of $G_n$ consisting of vertices of at most $L$-distance from $v$. 
%The spanning tree has the vertex $v$ as root and descendants of upto $\ell$ of upto $\ell$ generations from $v$.
%is a sequence of random vectors $(\V{N}_1(v), \ldots, \V{N}_L(v))$, where, $L$ is the diameter of the connected component of the graph $G_n$ containing $v$. 
The spanning tree is formed from the exploration process which starts from a vertex $v$ as the \emph{root} in the random graph $G_n$ generated from stochastic block model. The set of vertices of type $a$ of the random graph $G_n$  that are neighbors of $v$ and has not been previously explored are called $\Gamma_{1, a}(v)$ and $N_{1,a}(v) = |\Gamma_{1,a}(v)|$ for $a=1, \ldots, K$ and $N_1(v) = (N_{1,1}(v), \ldots, N_{1, K}(v))$. So, $\Gamma_1(v) = \{\Gamma_{1, 1}(v), \ldots, \Gamma_{1, K}(v)\}$ are the children of the root $v$ at step $\ell=1$ in the spanning tree of the neighborhood exploration process. The neighborhood exploration process is repeated at second step by looking at the neighbors of type $a$ of the vertices in $\Gamma_1(v)$ that has not been previously explored and the set is called $\Gamma_{2, a}(v)$ and $N_{2, a}(v) = |\Gamma_{2, a}(v)|$ for $a=1, \ldots, K$. Similarly, $\Gamma_2(v) = \{\Gamma_{2, 1}(v), \ldots, \Gamma_{2, K}(v)\}$ are the children of vertices $\Gamma_1(v)$ at step $\ell=2$ in the spanning tree of the neighborhood exploration process. The exploration process is continued until step $\ell = L$. Note that the process stops when all the vertices in $G_n$ has been explored. So, if $G_n$ is connected, then, $L \leq $ the diameter of the graph $G_n$.
%\end{definition}

Since, we either consider $G_n$ connected or only the giant component of $G_n$, the neighborhood exploration process will end in a finite number of steps but the number of steps may depend on $n$ and is equal to the diameter, $L$, of the connected component of the graph containing the root $v$. It follows from Theorem 14.11 of \cite{bollobas2007phase} that
\begin{align}
\label{eq_diameter}
L/\log_{\gl_1}(n) \stackrel{P}{\rar} 1.
\end{align}

%The zeroth generation of $\cb_B$ consists of a single particle of type $a$ is denoted by $\cb_B(a)$. 
%Also, the neighborhood exploration process $\cb_B$ is just the process $\cb_B(a)$ started with a single particle whose (random) type is distributed according to the probability measure $(\pi_1, \ldots, \pi_K)$.

%Let us recall our notation for the survival probabilities of particles in $\cb_B(a)$. We write $\rho_k(B; a)$ for the probability that the total population consists of exactly $k$ particles, and $\rho_{\geq k}(B; a)$ for the probability that the total population contains at least $k$ particles. Furthermore, $\rho(B; a)$ is the probability that the branching process survives for eternity.
%We write $\rho_k(B), \rho_{\geq k}(B)$ and $\rho(B)$ for the corresponding probabilities for $\cb_B$ , so that, e.g., $\rho_k(B)= \sum_{a = 1}^K \rho_k(B; a)\pi_a$.

Now, we find a coupling relation between the \textit{neighborhood exploration process} of a vertex of type $a$ in stochastic block model and a multi-type Galton-Watson process, $\cb(a)$ starting from a vertex of type $a$. The Lemma is based on Proposition 31 of \cite{bordenave2015non}.

%We assume all vertices of graph $G_n$ generated from a stochastic block model has been assigned a community or type $\xi_i$ (say) for vertex $v_i \in V(G_n)$. By \textit{neighborhood exploration process} of a vertex of type $a$ in stochastic block model, we mean that we start from a random vertex $v_i$ of type $a$ in the random graph $G_n$ generated from stochastic block model. Then, we count the number of vertices of the random graph $G_n$ are neighbors of $v_i$, $N(v_i)$. We repeat the neighborhood exploration process by looking at the neighbors of the vertices in $N(v_i)$. We continue until we have covered all the vertices in $G_n$. Since, we either consider $G_n$ connected or only the giant component of $G_n$, the neighborhood exploration process will end in finite steps but the number of steps may depend on $n$. 
\begin{lemma}
\label{lemma_brpr_sbm}
Let $w(n)$ be a sequence such that $w(n) \rar \infty$ and $w(n)/n \rar 0$. Let $(T, v)$ be the random rooted tree associated with the Poisson multi-type Galton-Watson branching process defined in Section \ref{sec_sbm} started from $Z_0 = \gd_{c_v}$ and $(G, v)$ be the spanning tree associated with neighborhood exploration process of random SBM graph $G_n$ starting from $v$. For $\ell \leq \tau$, where $\tau$ is the number of steps required to explore $w(n)$ vertices in $(G, v)$, the total variation distance, $d_{\mbox{TV}}$, between the law of $(G, v)_\ell$ and $(T, v)_\ell$ at step $\ell$ goes to zero as $O\left(n^{-\frac{1}{2}}\vee w(n)/n\right) = o(1)$.
%Let $(T, o)$ be the random rooted tree associated to the multi-type branching process defined in Section \ref{sec_sbm} started from $Z_0 = \gd_{\xi_v}$ and $(G,v)$ be the random rooted tree of the neighborhood exploration process for a stochastic block model graph with parameters $(P, \pi) = (B/n,\pi)$.  Then, within the giant component and for $\ell \sim \gkp\log_{\gl_1} n$ with $0 \leq \gkp < 1$, the total variation distance between the law of $(G, v)$ and $(T, o)$ at each time point point $t$ goes to 0 as $O\left(n^{-\gm\wedge(1-\gkp)}\right)$ under Assumption (A3). 
\end{lemma}
\begin{proof}
Let us start the neighborhood exploration process starting with vertex $v$ of a graph generated from an SBM model with parameters $(P, \pi) = (B/n,\pi)$. Correspondingly the multi-type branching process starts from a single particle of type $\V{c}_v$, where, $\V{c}_v$ is the type or class of vertex $v$ in SBM. 

Let $t$ be such that $0\leq t < \tau$, where, $\tau$ is defined in the Lemma statement.
%For $\ell \sim \gkp \log n$, we have with high probability $\tau \leq c\gl_1^\ell\log n$ from \cite{bordenave2015non} Lemma 29. 
Now, for such a $t\geq 0$, let $(x_{t+1}(1), \ldots, x_{t+1}(K))$ be leaves of $(T, v)$ at time $t$ starting from a vertex $v_{t}$ generated by step $t$ of class $\V{c}_{v_{t}} = a$. Let $(y_{t+1}(1), \ldots, y_{t+1}(K))$ be the vertices exposed at step $t$ of the exploration process starting from a vertex of class $a$, where, $a \in [K]$. Now, if $\V{c}_{v_{t}}$ is of type $a$, then, we have $x_{t+1}(b)$ follows $\mbox{Bin}(n_t(b), B_{ab}/n)$ and $y_{t+1}(b)$ follows $\mbox{Poi}(\pi_bB_{ab})$ for $b = 1, \ldots, K$, where, $n_{t}(b)$ is the number of unused vertices of type $b$ remaining at time $t$ for $b = 1, \ldots, K$. Also, $y_{t+1}(b)$ for different $b$ are independent. Note that $n_b \geq n_{t}(b)\geq n_b - w(n)$ for $b = 1, \ldots, K$. So, since, we have $|n_b/n - \pi_b| = O(n^{-1/2})$ for $b = 1, \ldots, K$, we get that,
\begin{align*}
|n_t(b) - \pi_b| < O\left(n^{-1/2} + w(n)/n\right)\ \ \ \mbox{ for } b = 1, \ldots, K
\end{align*}
%The number of new neighbors of $x$ of type $a$ has a binomial $\mbox{Bin}(n_t(b),B_{ab}/n)$ distribution, .

Now, we know that,
\begin{align*}
d_{TV}\left(\mbox{Bin}(m', \gl/m), \mbox{Poi}(m'\gl/m)\right) \leq \frac{\gl}{m}, \hspace{0.2in} d_{TV}\left(\mbox{Poi}(\gl), \mbox{Poi}(\gl')\right) \leq |\gl - \gl'|
\end{align*}
So, now, we have,
\begin{align*}
d_{TV}\left(P_{t+1}, Q_{t+1}\right) \leq O\left(n^{-1/2} \vee w(n)/n\right) = o(1)
\end{align*}
where, $P_{t+1}$ is the distribution of $y_{t+1}$ under neighborhood exploration process and $Q_{t+1}$ is the distribution of $x_{t+1}$ under the branching process, and hence Lemma \ref{lemma_brpr_sbm} follows.
%Now, following the proof of Proposition 31 in \cite{bordenave2015non} and Lemma 9.6 of \cite{bollobas2007phase}, we get that the total variation distance between distribution of $y_{t+1}$ and distribution of $x_{t+1}$ is of the order $O\left(n^{-\gm\wedge(1-\gkp)}\right)$.
\end{proof}

Now, we restrict ourselves to the giant component of $G_n$. The size of the giant component of $G_n$, $\mathcal{C}_1(G_n)$, of a random graph generated from SBM$(B, \pi)$ is related to the multi-type branching process through its survival probability as given in Definition \ref{def_survival_prob}. According to Theorem 3.1 of \cite{bollobas2007phase}, we have,
\begin{align}
\label{eq_giant_component}
\frac{1}{n}\mathcal{C}_1(G_n) \stackrel{P}{\rar} \rho(B, \pi)
\end{align}
%So, if we condition that the exploration process does not leave the giant component, it is same as conditioning that the branching process does not die out. 
Under this additional condition of restricting to the giant component, the branching process can be coupled with another branching process with a different kernel. The kernel of that branching process is given in following lemma.
\begin{lemma}
\label{lemma_brpr_cond}
If $v$ is in giant component of $G_n$, the new branching process has kernel $\left(B_{ab}\left(2\rho(B, \pi)/K - \rho^2(B, \pi)/K^2\right)\right)_{a,b = 1}^K$. 
\end{lemma}
\begin{proof}
The proof is given in Section 10 of \cite{bollobas2007phase}.
\end{proof}
Since, we will be restricting ourselves to the giant component of $G_n$, we shall be using the $B' \equiv \left(B_{ab}\left(2\rho(B, \pi)/K - \rho^2(B, \pi)/K^2\right)\right)_{a,b = 1}^K$ matrix as the connectivity matrix in stead of $B$. We abuse notation by referencing to the matrix $B'$ as $B$ too.

We proceed to prove the limiting behavior of typical distance between vertices $v$ and $w$ of $G_n$, where, $v, w\in V(G_n)$. We first try to find a lower bound for distance between two vertices. We shall separately give an upper bound and lower bounds for the distance between two vertices of the same type and different types. 

\begin{lemma}
\label{lm_low_bnd}
Under our model, for vertices $v,w\in V(G)$, if 
\begin{itemize}
\item[(a)] \hspace{0.05in} type of $v = $ type of $w = a$ (say), then, 
%\begin{align*}
%\E\left|\left\{\{v, w\}: d_G(v, w) \leq (1 - \gve)\tau_1\right\}\right| = O(n^{2-\gve})
%\end{align*}
%and so
\begin{align*}
\left|\left\{\{v, w\}: d_G(v, w) \leq (1 - \gve)\tau_1\right\}\right| \leq O(n^{2-\gve}) \mbox{ with high probability}
\end{align*}
where, $\tau_1$ is the minimum real positive $t$, which satisfies Eq. \eqref{eq_sbm_ev1},
%\begin{align}
%\label{eq_sbm_ev1}
%\left[\gl_2^t +\frac{\gl_1^t - \gl_2^t}{K}\right]Z_0(\gd_a) = n_a 
%\end{align}
\item[(b)] \hspace{0.05in} type of $v = a \neq b =$  type of $w$ (say), then, 
%\begin{align*}
%\E\left|\left\{\{v, w\}: d_G(v, w) \leq (1 - \gve)\tau_2\right\}\right| = O(n^{2-\gve})
%\end{align*}
%and so
\begin{align*}
\left|\left\{\{v, w\}: d_G(v, w) \leq (1 - \gve)\tau_2\right\}\right| \leq O(n^{2-\gve}) \mbox{ with high probability}
\end{align*}
where, $\tau_2$ is the minimum real positive $t$, which satisfies Eq. \eqref{eq_sbm_ev2}.
%\begin{align}
%\label{eq_sbm_ev2}
%\left[\gl_1^t - \gl_2^t\right]Z_0(\gd_a) = n_b
%\end{align}
\end{itemize}
\end{lemma}
\begin{proof}
%The proof is given in Appendix A3 and follows from Lemma 14.2 of \cite{bollobas2007phase}.
%We have $\s$ is finite, say $\s = \{1,2,\ldots,K\}$. 
Let $\G_d(v) \equiv \G_d(v, G_n)$ denote the $d$-distance set of $v$ in $G_n$, i.e., the set of vertices of $G_n$ at graph distance exactly $d$ from $v$, and let $\G_{\leq d}(v) \equiv \G_{\leq d}(v,G_n)$ denote the $d$-neighborhood $\cup_{d'\leq d}\G_{d'}(v)$ of $v$. Let $\G_{d,a}(v) \equiv \G_{d,a}(v, G_n)$ denote the set of vertices of type $a$ at $d$-distance in $G_n$ and let $\G_{\leq d,a}(v) \equiv \G_{\leq d, a}(v,G_n)$ denote the $d$-neighborhood $\cup_{d'\leq d}\G_{d', a}(v)$ of $v$ consisting of vertices of type $a$. Let $N^a_d$ be the number of particles at generation $d$ of the branching process $\cb_B(\gd_a)$ and $N^a_{d,c}$ be the number of particles at generation $d$ of the branching process $\cb_B(\gd_a)$ of type $c$. So, $N^a_d = \sum_{c=1}^K N^a_{d,c}$ and $Z_t(k) = \sum_{d=0}^tN^a_{d,k}$.

%Let $0 < \gve < 1/10$ be arbitrary. 
Lemma \ref{lemma_brpr_sbm} involved first showing that, for $n$ large enough, the neighborhood exploration process starting at a given vertex $v$ of $G_n$ with type $a$ could be coupled with the branching process $\cb_{B'} (\gd_a)$, where the $B'$ is defined by Lemma \ref{lemma_brpr_cond}. As noted we identify $B'$ with $B$. 

The neighborhood exploration process and multi-type branching process can be coupled so that for every $d$, $|\G_{d}(v)|$ is at most the number $N_d + O\left(n^{-\frac{1}{2}}\vee w(n)/n\right)$, where, $N_d$ is number of particles in generation $d$ of $\cb_{B}(\gd_a)$ and in $d$ generations at most $w(n)$ vertices of $G_n$ have been explored. 
%The number of vertices at generation $d$ of type $c$ of branching process $\cb_{B}(a)$, denoted by $N^a_{d,c}$ and the number of vertices of type $c$ at distance $d$ from $v$ for the neighborhood exploration process of $G_n$ is denoted by $|\G^a_{d, c}(v)|$, where, $c=1, \ldots, K$.

From Theorem \ref{thm_bp1}, we get that with high probability
\begin{align*}
\left|\frac{\langle\phi_k, Z_t\rangle}{\gl_k^t} - \langle\phi_k, Z_0\rangle\right| \leq C(t+1)^2(\log n)^{3/2}
\end{align*}
Since, for any $x\in \R^K$, we get the unique representation, $x = \sum_{k=1}^K \langle x, \phi_k\rangle \phi_k$, for any basis $\{\phi_k\}_{k=1}^K$ of $\R^K$. If we take $x = e_b$, where, $e_b$ is the unit vector with $1$ at $b$-th co-ordinate and $0$ elsewhere, $b = 1, \ldots, K$, we can get
\begin{align*}
Z_t(b) \leq \sum_{k=1}^K\phi_k(b)\gl_k^t\phi_k(a)\left[Z_0(a) + C(t+1)^2(\log n)^{3/2}\right] 
\end{align*}
Now, under our model one representation of the eigenvectors is $\phi_1 = \frac{1}{\sqrt{K}}(1,\ldots, 1)$, $\phi_2 = \frac{1}{\sqrt{2}}(-1, 1,0,\ldots, 0)$, $\phi_3 = \frac{1}{\sqrt{6}}(-1, -1, 2, 0, \ldots, 0)$, $\cdots$, \\ $\phi_{K-1} = \frac{1}{\sqrt{K(K-1)}}(-1,\ldots,-1,K-1)$.
Now using the representation of eigenvectors for branching process starting from vertex of type $a$, $a \in [K]$, we get with high probability
\begin{eqnarray*}
\sum_{k=1}^KZ_t(k) & \leq & \gl_1^t\left[Z_0(a) + C(t+1)^2(\log n)^{3/2}\right] \\
Z_t(a) - Z_t(b) & \geq &\gl_2^t\left[-Z_0(a) - C(t+1)^2(\log n)^{3/2}\right],\hspace{0.1in} b=1,\ldots,K\mbox{ and } b\neq a.
\end{eqnarray*}
So, we can simplify, for each $a \in [K]$ with $Z_0(a) = 1$, with high probability,
\begin{eqnarray*}
Z_t(a) & \leq & \frac{1}{K}\left(\gl_1^t+(K-1)\gl_2^t\right)\left[1 + C(t+1)^2(\log n)^{3/2}\right] \\
Z_t(b) & \leq &\frac{\gl_1^t - \gl_2^t}{K}\left[1 + C(t+1)^2(\log n)^{3/2}\right],\hspace{0.1in} b\in [K]\mbox{ and } b\neq a.
\end{eqnarray*}
% imply that $\E N_d = O\left( ||T_{(1+2\gve)K}||^d\right) = O(((1 + 2\gve)\gl)^d)$, where $\gl = ||T_{K}|| > 1$.

%Let $N^a_t(c)$ be the number of particles of type $c$ in the $t$-th generation of $\cb_K(a)$, then, $N^a_t$ is the vector $(N^a_t(1), \ldots, N^a_t (Q))$. Also, let $\nu = (\nu_1, \ldots, \nu_Q)$ be the eigenvector of $T_K$ with eigenvalue $\gl$ (unique, up to normalization, as $P$ is irreducible). From standard branching process results, we have
%\begin{align}
%\label{eq_brpr_lim1}
%N^a_t/\gl^t \stackrel{a.s.}{\rar} X\nu,
%\end{align}
%where $X \geq 0$ is a real-valued random variable, $X$ is continuous except that it has some mass at 0, and $X = 0$ if and only if the branching process eventually dies out and lastly,
%\begin{align*}
%\E X = \nu_a.
%\end{align*}
%under the conditions given in Theorem V.6.1 and Theorem V.6.2 of \cite{athreya1972branching}.
%So,
%\begin{align}
%\label{eq_brpr_elim1}
%N^a_t(b)/\gl^t \stackrel{L^1}{\rar} \nu_a\nu_b,
%\end{align}

Set $D_1 = (1 - \gve)\tau_1$, where, $\tau_1$ is the solution to the equation
\begin{align*}
\left[\gl_2^t +\frac{\gl_1^t - \gl_2^t}{K}\right] = n
\end{align*}
and set $D_2 = (1 - \gve)\tau_2$, where, $\tau_2$ is the solution to the equation
\begin{align*}
\left[\frac{\gl_1^t - \gl_2^t}{K}\right] = n
\end{align*}
where, $\gve>0$ is fixed and small. Note that both $\tau_1$ and $\tau_2$ are of the order $O(\log n)$. Thus, with high probability, for $v$ of type $a$ and $w(n) = O(n^{1-\gve})$,
\begin{eqnarray*}
|\G_{\leq D_1, a}(v)| & = \sum_{d = 0}^{D_1} N^a_{d,a} & \leq Z_{D_1}(a) + O\left(D_1n^{-\frac{1}{2}}\vee w(n)/n\right) = O(n^{1 - \gve}) \\
|\G_{\leq D_2, b}(v)| & = \sum_{d = 0}^{D_2} N^a_{d,b} & \leq Z_{D_2}(b) + O\left(D_2n^{-\frac{1}{2}}\vee w(n)/n\right) = O(n^{1 - \gve})
\end{eqnarray*}
So, summing over $v\in C_a$ and $v \in C_b$, where, $C_a = \{i\in V(G)| \V{c}_i = a\}$ and $C_b = \{i \in V(G)| \V{c}_i = b\}$, we have, 
\begin{eqnarray*}
\sum_{v\in C_a}|\G_{\leq D_1, a}(v)| & = & \left|\left\{\{v, w\}: d_G(v, w) \leq (1 - \gve)\tau_1, v,w\in C_a\right\}\right| \\
\sum_{v\in C_a}|\G_{\leq D_2, b}(v)| & = & \left|\left\{\{v, w\}: d_G(v, w) \leq (1 - \gve)\tau_2, v\in C_a, w\in C_b\right\}\right| 
\end{eqnarray*}
and so with high probability
\begin{eqnarray*}
\left|\left\{\{v, w\}: d_G(v, w) \leq (1 - \gve)\tau_1, v,w\in C_a\right\}\right| & = & \sum_{v\in V(G_n)}|\G_{\leq D, a}(v)| = O(n^{2-\gve}) \\
\left|\left\{\{v, w\}: d_G(v, w) \leq (1 - \gve)\tau_2, v\in C_a, w\in C_b\right\}\right| & = & \sum_{v\in V(G_n)}|\G_{\leq D, b}(v)| = O(n^{2-\gve})
\end{eqnarray*}
The above statement is equivalent to
%\begin{align*}
%\E\left|\left\{\{v, w\}: d_G(v, w) \leq (1 - \gve)\frac{\log n}{\log \gl/\log(\nu_a\nu_b)}\right\}\right| = \E\sum_{v\in V(G_n)}|\G_{\leq D}(v)| = O(n^{2-\gve})
%\end{align*}
%So, by Markov's Theorem, we have,
\begin{eqnarray*}
\Pr\left[\left|\left\{\{v, w\}: d_G(v, w) \leq (1 - \gve)\tau_1, v,w\in C_a\right\}\right|  \leq  O(n^{2-\gve}) \right] & = & 1 - o(1) \\
\Pr\left[\left|\left\{\{v, w\}: d_G(v, w) \leq (1 - \gve)\tau_2, v\in C_a, w\in C_b\right\}\right|  \leq  O(n^{2-\gve}) \right] & = & 1 - o(1)
\end{eqnarray*}
for any fixed $\gve > 0$. 
\end{proof}

%Now, we first try to upper bound the typical distance between two vertices of the same type. For the same type vertices, we just focus on the subgraph of the original graph from stochastic block model having vertices of same type. So, in Lemma \ref{lm_up_bnd1}, the graph $G_n$ is the subgraph of the original graph containing only the vertices of the same type in the final boundary. 
%\begin{lemma}
%\label{lm_up_bnd1}
%For vertices $v,w\in V(G)$, if type of $v = $ type of $w = a$ (say) 
%%and if $K_{aa}\pi_a > (K_{aa})^{\gm_a} > 1\ \forall a$ with $\gm = \min_a \gm_a$
%\begin{align*}
%\Pr\left(d_G(v, w) < (1 + \gve)\frac{\log (n\pi_a)}{2\log (\gl_2)}\right) = 1 - exp(-\Omega(n^{2\eta}))
%\end{align*}
%conditioned on the event that $u$ is in the giant component of $G$.
%\end{lemma}
%\begin{proof}
%The proof is given in Appendix A4 and follows from Lemma 14.3 of \cite{bollobas2007phase}.
%\end{proof}

Now, we upper bound the typical distance between two vertices of SBM graph $G_n$.
%different types. So, in Lemma \ref{lm_up_bnd2}, the graph $G_n$ is the original graph containing with vertices of the different types. So, the coupled branching process on that graph becomes a multi-type branching process.

\begin{lemma}
\label{lm_up_bnd2}
Under our model, for vertices $v,w\in V(G)$ and conditioned on the event that the exploration process starts from a vertex in the giant component of $G$, if, 
\begin{itemize}
\item[(a)] \hspace{0.05in} type of $v = $ type of $w = a$ (say), then, 
\begin{align*}
\Pr\left(d_G(v, w) < (1 + \gve)\tau_1\right) = 1 - exp(-\Omega(n^{2\eta}))
\end{align*}
where, $\tau_1$ is the minimum real positive $t$, which satisfies Eq. \eqref{eq_sbm_ev1},
\item[(b)] \hspace{0.05in} type of $v = a \neq b = $ type of $w$ (say), then, 
\begin{align*}
\Pr\left(d_G(v, w) < (1 + \gve)\tau_2\right) = 1 - exp(-\Omega(n^{2\eta}))
\end{align*}
where, $\tau_2$ is the minimum real positive $t$, which satisfies Eq. \eqref{eq_sbm_ev2}.
\end{itemize}
\end{lemma}
\begin{proof}
We consider the multi-type branching process with probability kernel $P_{ab}=\frac{B_{ab}}{n}$ $\forall a,b = 1, \ldots, K$ and the corresponding random graph $G_n$ generated from stochastic block model has in total $n$ nodes. We condition that branching process $\cb_{K}$ survives.

Note that an upper bound $1$ is obvious, since we are bounding a probability, so it suffices to prove a corresponding lower bound. We may and shall assume that $B_{ab} > 0$ for some  $a,b$. 

Again, let $\G_d(v) \equiv \G_d(v, G_n)$ denote the $d$-distance set of $v$ in $G_n$, i.e., the set of vertices of $G_n$ at graph distance exactly $d$ from $v$, and let $\G_{\leq d}(v) \equiv \G_{\leq d}(v,G_n)$ denote the $d$-neighborhood $\cup_{d'\leq d}\G_{d'}(v)$ of $v$. Let $\G_{d,a}(v) \equiv \G_{d,a}(v, G_n)$ denote the set of vertices of type $a$ at $d$-distance in $G_n$ and let $\G_{\leq d,a}(v) \equiv \G_{\leq d, a}(v,G_n)$ denote the $d$-neighborhood $\cup_{d'\leq d}\G_{d', a}(v)$ of $v$ consisting of vertices of type $a$. Let $N^a_d$ be the number of particles at generation $d$ of branching process $\cb_B(\gd_a)$ and $N^a_{d,c}$ be the number of particles at generation $d$ of branching process $\cb_B(\gd_a)$ of type $c$. So, $N^a_d = \sum_{c=1}^K N^a_{d,c}$ and $Z_t(k) = \sum_{d=0}^tN^a_{d,k}$.

%Fix $0 < \eta < 1/10$. We shall assume that $\eta$ is small enough that $(1 - 2\eta)\gl > 1$. In the argument leading to \eqref{eq_brpr_bound1} in proof of Lemma \ref{lemma_brpr_sbm}, we showed that, given $\om(n)$ with $\om(n) = o(n)$ and a vertex $v$ of type $a$, the neighborhood exploration process of $v$ in $G_n$ could be coupled with the branching process $\cb_{(1-2\eta)K}(a)$ so that whp the former dominates until it reaches size $\om(n)$. More precisely, writing $N_{d,c}$ for the number of particles of type $c$ in generation $d$ of $\cb_{(1-2\eta)K}(a)$, and $\G_{d,c}(v)$ for the set of type $c$ vertices at graph distance $d$ from $v$, whp 
By Lemma \ref{lemma_brpr_sbm}, for $w(n) = o(n)$,
\begin{align}
\label{eq_brpr_count1}
|\G_{d,c}(v)| \geq N_{d,c} - O\left(n^{-\frac{1}{2}}\vee w(n)/n\right),\ c = 1,\ldots,K.
\end{align}
for all $d$ s.t. $|\G_{\leq d}(v)| < \om(n)$. This relation between the number of vertices at generation $d$ of type $c$ of branching process $\cb_{B}(\gd_a)$, denoted by $N_{d,c}$ and the number of vertices of type $c$ at distance $d$ from $v$ for the neighborhood exploration process of $G_n$, denoted by $|\G_{d, c}(v)|$ becomes highly important later on in this proof, where, $c=1, \ldots, K$. Note that the relation only holds when $|\G_{\leq d}(v)| < \om(n)$ for some $\om(n)$ such that $\om(n)/n \rar 0$ as $n\rar \infty$.

%Let us call a kernel $K$ bipartite if $S=L\cup R$, with $K_{ab}=0$ whenever $a,b\in L$ or $a,b\in R$, in which case the graph $G_n$ is bipartite. We shall suppose that $K$ is not bipartite. 

%Let $N^a_{d,c}$ be the number of particles of type $c$ in the $d$-th generation of $\cb_B(\gd_a)$, then, $N^a_d$ is the vector $(N^a_d(1), \ldots, N^a_d(K))$. Also, let $\nu = (\nu_1, \ldots, \nu_K)$ be the eigenvector of $S$ with eigenvalue $\gl_1$ (unique, up to normalization, as $P$ is irreducible). From standard branching process results, we have
%\begin{align}
%\label{eq_brpr_lim1}
%N^a_t/\gl^t \stackrel{a.s.}{\rar} X\nu,
%\end{align}
%where $X \geq 0$ is a real-valued random variable, $X$ is continuous except that it has some mass at 0, and $X = 0$ if and only if the branching process eventually dies out and lastly,
%\begin{align*}
%\E X = \nu_a
%\end{align*}
%under the conditions given in Theorem V.6.1 and Theorem V.6.2 of \cite{athreya1972branching}.
%So,
%\begin{align}
%\label{eq_brpr_elim1}
%N^a_t(b)/\gl^t \stackrel{L^1}{\rar} \nu_a\nu_b,
%\end{align}

From Theorem \ref{thm_bp1} of the branching process, we get that with high probability
\begin{align*}
\left|\frac{\langle\phi_k, Z_t\rangle}{\gl_k^t} - \langle\phi_k, Z_0\rangle\right| \leq C(\log n)^{3/2}
\end{align*}
%So, we get that, with high probability
%\begin{eqnarray*}
%\sum_{k=1}^KZ_t(k) & \geq & \gl_1^t\left[1 - C(t+1)^2(\log n)^{3/2}\right] \\
%Z_t(\gd_a) - Z_t(\gd_b) & \geq &\gl_2^t\left[1 - C(t+1)^2(\log n)^{3/2}\right],\hspace{0.1in} b=1,\ldots,K\mbox{ and } a\neq b
%\end{eqnarray*}

Now following the same line of argument as in proof of Lemma \ref{lm_low_bnd}, for each $a \in [K]$ with $Z_0(a) = 1$, with high probability we get that,
\begin{eqnarray*}
Z_t(a) & \leq & \frac{1}{K}\left(\gl_1^t+(K-1)\gl_2^t\right)\left[1 + C(t+1)^2(\log n)^{3/2}\right] \\
Z_t(b) & \leq &\frac{\gl_1^t - \gl_2^t}{K}\left[1 + C(t+1)^2(\log n)^{3/2}\right],\hspace{0.1in} b\in [K]\mbox{ and } b\neq a.
\end{eqnarray*}

Let $D_1$ be the integer part of $(1+2\eta)\tau'_1$, where, $\tau'_1$ is the solution to the equation
\begin{align}
\label{eq_eq1}
\left[\gl_2^t +\frac{\gl_1^t - \gl_2^t}{K}\right] = n^{1/2 - \eta} 
\end{align}
Thus conditioned on survival of the branching process $\cb_{B}(\gd_a)$, $N^a_{D_1,a} \geq n^{1/2+\eta/2}$. 
%(note that $N^a_{D,a}$ comes from branching process $\cb_{B}(\gd_a)$ not branching process $\cb_{K}(a)$). 
Set $D_2 = (1 + \eta)\tau'_2$, where, $\tau'_2$ is the solution to the equation
%\begin{align*}
%\left[\gl_1^t - \gl_2^t\right]Z_0(\gd_a) = n_b^{1/2+\eta}
%\end{align*}
\begin{align}
\label{eq_eq2}
\gl_1^t = n^{1/2+\eta}
\end{align}
Thus conditioned on survival of branching process $\cb_{B}(\gd_a)$, $N^a_{D_2,b} \geq n^{1/2+\eta/2}$ for $b=1,\ldots, K$.
%(note that $N^a_{D,a}$ comes from branching process $\cb_{B}(\gd_a)$ not branching process $\cb_{K}(a)$). 
Furthermore $\lim_{d \rar \infty} \Pr(N^a_d \neq 0) = \rho(B, a)$.
%where, $\gve>0$ is fixed and small. Thus, with high probability, for $v$ of type $a$

Now, we have conditioned that the branching process with kernel $B$ is surviving. The right-hand side tends to $\rho(B, a) = 1$ as $\eta \rar 0$. Hence, given any fixed $\gm > 0$, if we choose $\eta > 0$ small enough, and for large enough $n$, we have
%\begin{eqnarray*}
%\Pr\left(\forall b\neq a: N^a_{D_2,b} \geq n^{1/2+\eta/2}\right) & \geq & \rho(B, a)-\gm,\\
%\Pr\left(N^a_{D_1,a} \geq n_a^{1/2+\eta/2}\right) & \geq & \rho(B, \gd_a)- \gm.
%\end{eqnarray*}
\begin{eqnarray*}
\Pr\left(\forall b: N^a_{D_2,b} \geq n^{1/2+\eta/2}\right) & = & 1,\\
\Pr\left(N^a_{D_1,a} \geq n^{1/2+\eta/2}\right) & = & 1.
\end{eqnarray*}

Now, the neighborhood exploration process and branching process can be coupled so that for every $d$, $|\G_{d}(v)|$ is at most the number $N_d$ of particles in generation $d$ of $\cb_{B}(a)$ from Lemma \ref{lemma_brpr_sbm} and Eq \eqref{eq_brpr_count1}. So, we have for $v$ of type $a$, with high probability,
\begin{eqnarray*}
|\G_{\leq D_1, a}(v)| & \leq & \E\sum_{d = 0}^{D_1} N_d = o(n^{2/3}) \\
|\G_{\leq D_2, b}(v)| & \leq & \E\sum_{d = 0}^{D_2} N_d = o(n^{2/3})
\end{eqnarray*}
if $\eta$ is small enough, since $D_1$ is integer part of $(1+2\eta)\tau'_1$ and $D_2$ is the integer part of $(1+2\eta)\tau'_2$, where, $\tau'_1$ and $\tau'_2$ are solutions to Eq. \eqref{eq_eq1} and \eqref{eq_eq2}. Note that the power $2/3$ here is arbitrary, we could have any power in the range $(1/2, 1)$. 
%Hence,
%\begin{align*}
%|\G_{\leq D}(v)| \leq n^{2/3}\ \ whp,
%\end{align*} 
%and whp the coupling described in \eqref{eq_brpr_count1} extends at least to the $D$-neighborhood. 
So, now, we are in a position to apply Eq \eqref{eq_brpr_count1}, as we have $|\G_{\leq D}(v)| \leq O(n_a^{2/3}) < \om(n)$, with $\om(n)/n \rar 0$.

Now let $v$ and $w$ be two fixed vertices of $G(n,P)$, of types $a$ and $b$ respectively. We explore both their neighborhoods at the same time, stopping either when we reach distance $D$ in both neighborhoods, or we find an edge from one to the other, in which case $v$ and $w$ are within graph distance $2D + 1$. We consider two independent branching processes $\cb_{B}(a)$, $\cb'_{B}(b)$, with $N^a_{d, c}$ and $N^b_{d, c}$ vertices of type $c$ in generation $d$ respectively. By the previous argument, with high probability we encounter $o(n)$ vertices in the exploration so, by the argument leading to \eqref{eq_brpr_count1}, whp either the explorations meet, or 
\begin{eqnarray*}
|\G^a_{d,c}(w)| &\geq& Z^{(a)}_d(c) - O\left(n^{-\frac{1}{2}}\vee n^{-\frac{1}{3}}\right),\ c = 1,\ldots,K, c\neq a\\
|\G^b_{d,c}(w)| &\geq& Z^{(b)}_d(c) - O\left(n^{-\frac{1}{2}}\vee n^{-\frac{1}{3}}\right),\ c = 1,\ldots,K, c\neq b
\end{eqnarray*}
with the explorations not meeting, where, $Z^{(a)}$ is the branching process starting from $Z_0 = \gd_a$, for $a=1, \ldots, K$. Using bound on $N^a_{d,c}$ and the independence of the branching processes, it follows that for $a=b$, 
\begin{align*}
\Pr\left(d(v,w) \leq 2D_1+1 \mbox{ or } |\G^a_{D_1,c}(v)|, |\G^a_{D_1,c}(w)| \geq n^{1/2+\eta}\right) \geq 1 - o(1). 
\end{align*}
and for $a \neq b$,
\begin{align*}
\Pr\left(d(v,w) \leq 2D_2+1 \mbox{ or } \forall c :|\G^a_{D_2,c}(v)|, |\G^b_{D_2,c}(w)| \geq n^{1/2+\eta}\right) \geq 1 - o(1). 
\end{align*}
Write these probabilities as $\Pr(A_j \cup B_j)$, $j =1,2$. We now show that $\Pr(A_j^c\cap B_j) \rar 0$ and since $\Pr(A_j \cup B_j) \rar 1$, we will have $\Pr(A_j) \rar 1$.
%Note that the two events in the above probability statements are not disjoint. We shall try to find the probability that the second event in the above equation holds but not the first. 
We have not examined any edges from $\G_D(v)$ to $\G_D(w)$, so these edges are present independently with their original unconditioned probabilities. 
%For the first probability, both end-vertices being of type $a$, the expected number of these edges is at least $|\G^a_{D,a}(v)||\G^a_{D,a}(w)|B_{aa}/n_a$. This expectation is $\Omega((n_a^{1/2+\eta/2})^2/n_a) = \Omega(n_a^{\eta})$. It follows that at least one edge is present with probability $1 - \exp(-\Omega(n_a^{2\eta})) = 1 - o(1)$. Again, f
For any end vertex types $c_1$, $c_2$, the expected number of these edges is at least $|\G^a_{D,c}(v)||\G^a_{D,c}(w)|B_{c_1c_2}/n$ for first probability and $|\G^a_{D,c_1}(v)||\G^b_{D,c_2}(w)|B_{c_1c_2}/n$ for second probability. Choosing $c_1, c_2$ such that $B_{c_1c_2} > 0$, this expectation is $\Omega((n^{1/2+\eta/2})^2/n) = \Omega(n^{\eta})$. It follows that at least one edge is present with probability $1 - \exp(-\Omega(n^{\eta})) = 1 - o(1)$. If such an edge is present, then $d(v, w) \leq 2D_1 + 1$ for first probability and $d(v, w) \leq 2D_1 + 1$ for second probability. So, the probability that the second event in the above equation holds but not the first is $o(1)$. Thus, the last equation implies that
\begin{eqnarray*}
\Pr(d(v,w) \leq 2D_1+1) &\geq& (1 - \gm)^2 - o(1) \geq 1 - 2\gm - o(1)\\
\Pr(d(v,w) \leq 2D_2+1) &\geq& (1 - \gm)^2 - o(1) \geq 1 - 2\gm - o(1).
\end{eqnarray*}
where, $\gm > 0$ is arbitrary. Choosing $\eta$ small enough, we have $2D + 1 \leq (1 + \gve)\log (n)/\log \gl$. As $\gm$ is arbitrary, we
have
\begin{eqnarray*}
\Pr(d(v,w) \leq (1+\gve)\tau_1) &\geq& 1 - \exp(-\Omega(n^{2\eta})),\\
\Pr(d(v,w) \leq (1+\gve)\tau_2) &\geq& 1 - \exp(-\Omega(n^{2\eta})).
\end{eqnarray*}
and the lemma follows.
\end{proof}

The equations \eqref{eq_sbm_ev1} and \eqref{eq_sbm_ev2} control the asymptotic bounds for the graph distance $d_G(v, w)$ between two vertices $v$ and $w$ in $V(G_n)$. Under the condition (A3) it follows that $\gl_2^2>\gl_1$. If we consider $\gl_2^2 = c\gl_1$, where, $c$ is a constant, then the equations \eqref{eq_sbm_ev1} and \eqref{eq_sbm_ev2} can be written in the form of quadratic equations. So, the solutions $\tau_1$ and $\tau_2$ exist under the condition $c^{\tau_1}$ and $c^{\tau_2}$ are of the order $O(n)$ and the resulting solutions $\tau_1$ and $\tau_2$ are both of the order $O(\log n)$. Also, from the expression of the solutions $\tau_1$ and $\tau_2$, the limits $\frac{\tau_1}{\log n}$ and $\frac{\tau_2}{\log n}$ exist and we shall define the limit as $\gs_1$ and $\gs_2$ respectively.  

\subsection{Proof of Theorem \ref{thm_geo_dis_bnd} and Theorem \ref{thm_geo_dis}}
\label{sec_det_geo}
\subsubsection{Proof of Theorem \ref{thm_geo_dis_bnd}}
We shall try to prove the limiting behavior of the typical graph distance in the giant component as $n \rar \infty$. The Theorem essentially follows from Lemma \ref{lm_low_bnd} - \ref{lm_up_bnd2}. Under the conditions mentioned in the Theorem, part (a) follows from Lemma \ref{lm_low_bnd}(a) and \ref{lm_up_bnd2}(a) and part (b) follows from Lemma \ref{lm_low_bnd}(b) and \ref{lm_up_bnd2}(b).
\subsubsection{Proof of Theorem \ref{thm_geo_dis}}
From Definition \ref{def_geo_dis}, we have that $\bld_{ij} = $ graph distance between vertices $v_i$ and $v_j$, where, $v_i, v_j \in V(G_n)$.
%\begin{itemize}
%\item[Case 1:] For the case when type of $v_i = $ type of $v_j = a$ (say)\\
%From Lemma \ref{lm_low_bnd}(b), we get for any vertices $v$ and $w$ of same type $a$ with high probability, 
%\begin{align*}
%\left|\left\{\{v, w\}: d_G(v, w) \leq (1 - \gve)\frac{\log n}{\log (\pi_aB_{aa})}\right\}\right| \leq O(n^{2-\gve}).
%\end{align*}
%Also from Lemma \ref{lm_up_bnd1}, we get, for any vertices $v$ and $w$ of same type $a$
%\begin{align*}
%\Pr\left(d_G(v, w) < (1 + \gve)\frac{\log n}{\log (\pi_aB_{aa})}\right) = 1 - exp(-\Omega(n^{2\eta}))
%\end{align*}
%So, we have that, for $v_i, v_j$ having same type $a$, with high probability,
%\begin{align*}
%\left(\frac{\bld_{ij}}{\log n} - \D_{ij}\right)^2 \leq \mbox{Constant }\gve^2
%\end{align*}
%Since, $\gre = O(n^{-1/2})$ by Eq \eqref{eq_gre} and $(1 - exp(-\Omega(n^{2\eta})))^{n^2} \rar 1$ as $n\rar \infty$,
%\begin{align*}
%\sum_{i,j = 1: type(v_i)=type(v_j)}^n\left(\frac{\bld_{ij}}{\log n} - \D_{ij}\right)^2 \leq \gve^2.O((n\pi_a)^2) = O(n)
%\end{align*}
%with high probability.
%
%\item[Case 2:] 
%For the case when type of $v_i \neq $ type of $v_j $ \\
From Lemma \ref{lm_low_bnd}, we get for any vertices $v$ and $w$ with high probability, 
\begin{eqnarray*}
\left|\left\{\{v, w\}: d_G(v, w) \leq (1 - \gve)\tau_1\right\}\right| \leq O(n^{2-\gve}), \mbox{ if type of }v = \mbox{ type of } w\\
\left|\left\{\{v, w\}: d_G(v, w) \leq (1 - \gve)\tau_2\right\}\right| \leq O(n^{2-\gve}), \mbox{ if type of }v \neq \mbox{ type of } w.
\end{eqnarray*}
Also, from Lemma \ref{lm_up_bnd2}, we get
\begin{eqnarray*}
\Pr\left(d_G(v, w) < (1 + \gve)\tau_1\right) = 1 - \exp(-\Omega(n^{2\eta})), \mbox{ if type of }v = \mbox{ type of } w,\\
\Pr\left(d_G(v, w) < (1 + \gve)\tau_2\right) = 1 - \exp(-\Omega(n^{2\eta})), \mbox{ if type of }v = \mbox{ type of } w.
\end{eqnarray*}
Now, $\gs_1 = \tau_1/\log n$ and $\gs_2 = \tau_2/\log n$ are asymptotically constant as both $\tau_1$ and $\tau_2$ are of the order $\log n$ as follows from equations \eqref{eq_sbm_ev1} and \eqref{eq_sbm_ev2}. So, putting the two statements together, we get that with high probability,
\begin{align*}
\sum_{i,j = 1: type(v_i)\neq type(v_j)}^n\left(\frac{\bld_{ij}}{\log n} - \D_{ij}\right)^2 =  O(n^{2-\gve}) + O(n^2).\gve^2 
\end{align*}
since, by Lemma \ref{lemma_brpr_sbm}, $\gre = o(1)$ and $(1 - \exp(-\Omega(n^{2\eta})))^{n^2} \rar 1$ as $n\rar \infty$.
%\end{itemize}
So, putting the two cases together, we get that with high probability, for some $\gve > 0$,
\begin{align*}
\sum_{i,j = 1}^n\left(\frac{\bld_{ij}}{\log n} - \D_{ij}\right)^2 =  O(n^{2-\gve}) + O(n^2).\gve^2 = o(n^{2}). 
\end{align*}
Hence, for some $\gve > 0$,
\begin{align*}
\left|\left|\frac{\bld}{\log n} - \D\right|\right|_F \leq o(n).
\end{align*}

We have completed proofs of Theorems \ref{thm_geo_dis_bnd} and \ref{thm_geo_dis}.

\subsection{Perturbation Theory of Linear Operators}
\label{sec_det_pert}
We now establish part II of our program. $D$ can be considered as a perturbation of the operator $\D$.

The Davis-Kahan Theorem \cite{davis1970rotation}] gives a bound on perturbation of eigenspace instead of eigenvector, as discussed previously. 
\begin{theorem}[Davis-Kahan (1970)\cite{davis1970rotation}]
\label{thm_dk}
Let $\bh, \bh' \in \R^{n\times n}$ be symmetric, suppose $\cV \subset \R$ is an interval, and suppose for some positive integer $d$ that $\bw, \bw' \in \R^{n\times d}$ are such that the columns of $\bw$ form an orthonormal basis for the sum of the eigenspaces of $\bh$ associated with the eigenvalues of $\bh$ in $\cV$ and that the columns of $\bw'$ form an orthonormal basis for the sum of the eigenspaces of $\bh'$ associated with the eigenvalues of $\bh'$ in $\cV$. Let $\gd$ be the minimum distance between any eigenvalue of $\bh$ in $\cV$ and any eigenvalue of $\bh$ not in $\cV$ . Then there exists an orthogonal matrix $\br \in \R^{d\times d}$ such that $||\bw\br - \bw'||_{F} \leq \sqrt{2}\frac{||\bh - \bh'||_{F}}{\gd}$.
\end{theorem}

\subsection{Proof of Theorem \ref{thm_mis}} 
\label{sec_proof_thm}
The behavior of the eigenvalues of the limiting operator $\D$ can be stated as follows -
\begin{lemma}
Under our model, the eigenvalues of $\D$ - $|\mu_1(\D)| \geq |\mu_2(\D)| \geq\cdots \geq|\mu_n(\D)|$, can be bounded as follows -
\begin{align}
\label{eq_eigvl_rel}
\mu_1(\D) = O(n\gs_1),\ |\mu_K(\D)| = O(n(\gs_1-\gs_2)),\ \mu_{K+1}(\D) = \cdots = \mu_{n}(\D) = -\gs_1
\end{align}
%where, $\tilde{d}$, a vector of length $K$, is defined in Eq. \eqref{eq_lim_rel} and the smallest $(n-K)$ absolute eigenvalues of $\D$ are $-\tilde{d}$ where $-\tilde{d}_a$ has multiplicity $(n_a-1)$ for $a=1, \ldots, K$.

Also, With high probability it holds that $|\mu_K(\bld/\log n)| = O(n(\gs_1 - \gs_2))$ and \\ 
$\mu_{K+1}(\bld/\log n) \leq o(n)$.
\end{lemma} 
\begin{proof}
The matrix $\D + \gs_1\mbox{I}_{n\times n}$ is a block matrix with blocks of sizes $\{n_a\}_{a=1}^K$, with $\sum_{a=1}^K n_a = n$. The elements of $(a,b)$th block are all same and equal to $\gs_1$, if $a=b$ and equal to $\gs_2$, if $a\neq b$. Note, diagonal of $\D$ is zero, as diagonal of $\bld$ is also zero. Now, we have the eigenvalues of the $K\times K$ matrix of the values in $\D$ to be $(\gs_1+(K-1)\gs_2, \gs_1-\gs_2, \ldots, \gs_1 - \gs_2)$. If we consider, $\gl_2^2 = c\gl_1$, then, if $c > 1$, we will have $\gs_1 > \gs_2$. So, under our model, we have that $\gs_1 > \gs_2$. So, because of repetitions in the block matrix $\mu_1(\D) = O(n\gs_1) = O(n)$ and $\mu_K(\D) = O(n(\gs_1-\gs_2)) = O(n)$, since, by assumption (A3), $n_a = O(n)$, for all $a=1,\ldots, K$. Now, the rest of the eigenvalues of $\D + \gs_1\mbox{Id}_{n\times n}$ is zero, so the rest of eigenvalues of $\D$ is $-\gs_1$.
%The matrix $\D$ can be considered as a Khatri-Rao product of the matrices $\cd$ and $\bj$ according to equation \eqref{eq_lim_rel}. Now, there exists a constant $\tau$ such that $\log||T_B||>\tau>0$, since $||T_B||>1$. So, we have $\gl_1(\cd) < \tau$. So, we have $\mu_1(\cd) < 1$ and since $n_a \leq n$ for all $a$ and $\sum_an_a = n$. So, we have $\gl_1(\D)\leq n$. Now, By Assumption (C2) and (C4), $\gl_K(\cd) \geq \ga$ and $n_a \geq \gm n$, so, $\gl_K(\D)\geq \ga\gm n$. Now, it is easy to see that the remaining eigenvalues of $\D$ is -1, since, $\cb\star\bj$ is a rank $K$ matrix and its remaining eigenvalues are zero and the eigenvalues of diagonal matrix are $\tilde{d}$ with $\tilde{d}_a$ having multiplicity $(n_a)$ for $a=1, \ldots, K$.

Now, about the second part of Lemma, By Weyl's Inequality, for all $i=1, \ldots, n$,
\begin{eqnarray*}
||\mu_i(\bld/\log n)| - |\gl_i(\D)|| & \leq & \left|\left|\bld/\log n - \D\right|\right|_F \leq o(n) 
% & \leq & O(n^{1-\gve})
\end{eqnarray*}
Since, from (A1)-(A3), it follows that $\gs_1 - \gs_2 > c >0$, for some constant $c$, so, $|\gl_K(\bld/\log n)| = O(n(\gs_1 - \gs_2)) - o(n) = O(n(\gs_1 - \gs_2))$ for large $n$ and $|\gl_{K+1}(\bld/\log n)| \leq -\gs_1 + o(n) = o(n)$.
\end{proof}
%\begin{corollary}
%\label{cor_eig}
%With high probability it holds that $|\gl_K(\bld/\log n)| \geq O(n)$ and \\ 
%$\gl_{K+1}(\bld/\log n) \leq O(n^{1-\gve})$.
%\end{corollary}
%\begin{proof}
%By Weyl's Inequality, for all $i=1, \ldots, n$,
%\begin{eqnarray*}
%||\gl_i(\bld/\log n)| - |\gl_i(\D)|| & \leq & \left|\left|\frac{\bld}{\log n} - \D\right|\right|_F \leq O(n^{1-\gve/2}) \\
% & \leq & O(n^{1-\gve})
%\end{eqnarray*}
%So, $|\gl_K(\bld/\log n)| \geq O(n) - O(n^{1-\gve}) = O(n)$ for large $n$ and $|\gl_{K+1}(\bld/\log n)| \leq -1 + O(n^{1-\gve}) = O(n^{1-\gve})$.
%\end{proof}
Now, let $\bw$ be the eigenspace corresponding to the top $K$ absolute eigenvalues of $\D$ and $\tilde{\bw}$ be the eigenspace corresponding to the top $K$ absolute eigenvalues of $\bld$. Using Davis-Kahan
\begin{lemma}
\label{lm_eigsp}
With high probability, there exists an orthogonal matrix $\br\in\R^{K\times K}$ such that $||\bw\br - \tilde{\bw}||_{F} \leq o\left((\gs_1 - \gs_2)^{-1}\right)$
\end{lemma}
\begin{proof}
The top $K$ eigenvalues of both $\D$ and $\bld/\log n$ lies in $(Cn, \infty)$ for some $C>0$. Also, the gap $\gd = O(n(\gs_1 - \gs_2))$ between top $K$ and $K+1$th eigenvalues of matrix $\D$. So, now, we can apply Davis-Kahan Theorem \ref{thm_dk} and Theorem \ref{thm_geo_dis}, to get that,
\begin{align*}
||\bw\br - \tilde{\bw}||_F \leq \sqrt{2}\frac{\left|\left|\bld/\log n - \D\right|\right|_F}{\gd} \leq \frac{o(n)}{O(n(\gs_1 - \gs_2))} = o\left((\gs_1 - \gs_2)^{-1}\right)
\end{align*} 
\end{proof}

Now, the relationship between the rows of $W$ can be specified as follows -
\begin{lemma}
\label{lm_rows}
For any two rows $i,j$ of $\bw_{n\times K}$ matrix, $||u_i - u_j||_2 \geq O(1/\sqrt{n})$, if type of $v_i\neq $ type of $v_j$.
\end{lemma} 
\begin{proof}
%The matrix $\D$ can be considered as a Khatri-Rao product of the matrices $\cd$ and $\bj$ according to equation \eqref{eq_lim_rel}. Now, by Assumption (C3), we have a constant difference between the rows of matrix $\cd$. So, rows of $\D$ as well as the projection of $\D$ into into its top $K$ eigenspace has difference of order $O(n^{-1/2})$ between rows of matrix. 
The matrix $\D + \gs_1\mbox{Id}_{n\times n}$ is a block matrix with blocks of sizes $\{n_a\}_{a=1}^K$, with $\sum_{a=1}^K n_a = n$. The elements of $(a,b)$th block are all same and equal to $\gs_1$, if $a=b$ and equal to $\gs_2$, if $a\neq b$. Note, diagonal of $\D$ is zero, as diagonal of $\bld$ is also zero. Now, we have the rows of eigenvectors of the $K\times K$ matrix of the values in $\D$ that have a constant difference. Under our model, we have that $\gs_1 > \gs_2$. So, because of repetitions in the block matrix, rows of $\D$ as well as the projection of $\D$ into into its top $K$ eigenspace has difference of order $O(n^{-1/2})$ between rows of matrix. 
\end{proof}

Now, if we consider $K$-means criterion as the clustering criterion on $\tilde{\bw}$, then, for the $K$-means minimizer centroid matrix $\bc$ is an $n\times K$ matrix with $K$ distinct rows corresponding to the $K$ centroids of $K$-means algorithm. By property of $K$-means objective function and Lemma \ref{lm_eigsp}, with high probability,
\begin{eqnarray*}
||\bc - \tilde{\bw}||_F & \leq & ||\bw\br - \tilde{\bw}||_F\\
||\bc - \bw\br||_F & \leq & ||\bc - \tilde{\bw}||_F + ||\bw\br - \tilde{\bw}||_F \\
||\bc - \bw\br||^2_F & \leq & 4||\bw\br - \tilde{\bw}||^2_F \\
 & \leq & o\left((\gs_1 - \gs_2)^{-2}\right) \\
\end{eqnarray*}

By Lemma \ref{lm_rows}, for large $n$, we can get constant $C$, such that, $K$ balls, $B_1, \ldots, B_K$, of radius $r = Cn^{-1/2}$ around $K$ distinct rows of $\bw$ are disjoint.

Now note that with high probability the number of rows $i$ such that $||\bc_i - (\bw\br)_i|| > r$ is at most $\frac{cn}{(\gs_1 - \gs_2)^2}$, with arbitrarily small constant $c > 0$. If the statement does not hold then,
\begin{eqnarray*}
||\bc-\bw\br||^2_F & > & r^2.\left(\frac{cn}{(\gs_1 - \gs_2)^2}\right) \\
 & \geq & Cn^{-1}.\left(\frac{cn}{(\gs_1 - \gs_2)^2}\right) = O\left((\gs_1 - \gs_2)^{-2}\right)
\end{eqnarray*}
So, we get a contradiction, since $||\bc - \bw\br||^2_F \leq o\left((\gs_1 - \gs_2)^{-2}\right)$. Thus, the number of mistakes should be at most $\left(\frac{cn}{(\gs_1 - \gs_2)^2}\right)$, with arbitrarily small constant $c > 0$. 

So, for each $v_i \in V(G_n)$, if $\V{c}(v_i)$ is the type of $v_i$ and $\hat{\V{c}}(v_i)$ is the type of $v_i$ as estimated from applying $K$-means on top $K$ eigenspace of geodesic matrix $\bld$, we get that for arbitrarily small constant, $c > 0$, 
\begin{align*}
\left[\frac{1}{n}\sum_{i=1}^n\mathbf{1}\left(\V{c}(v_i) \neq \hat{\V{c}}(v_i)\right) < \frac{c}{(\gs_1 - \gs_2)^2}\right] \rar 1
\end{align*}
So, for constant $\gs_1$ and $\gs_2$, we get $c > 0$ such that, 
\begin{align*}
\left[\frac{1}{n}\sum_{i=1}^n\mathbf{1}\left(\V{c}(v_i) \neq \hat{\V{c}}(v_i)\right) < \frac{1}{2}\right] \rar 1
\end{align*}
%with high probability, for some small $0<\gve$,
%\begin{align*}
%\min_{\pi\in\mathcal{P}_K} |\{u \in V: \xi(u)\neq \pi(\hat{\xi}(u))\}| = O(n^{1-2\gve})
%\end{align*}
%where, $n$ is the number of vertices in the giant component.

\section{Conclusion}
\label{conclusion}
We have given an overview of spectral clustering in the context of community detection of networks and clustering. We have also introduced a new method of community detection in the paper and we have shown bounds on theoretical performance of the method. 
%The method also performs equally well empirically and we shall demonstrate that in follow-up papers.

%
%\begin{acknowledgement}
%If you want to include acknowledgments of assistance and the like at the end of an individual chapter please use the \verb|acknowledgement| environment -- it will automatically render Springer's preferred layout.
%\end{acknowledgement}
%
%\section*{Appendix}
%\addcontentsline{toc}{section}{Appendix}
%%
%%
%When placed at the end of a chapter or contribution (as opposed to at the end of the book), the numbering of tables, figures, and equations in the appendix section continues on from that in the main text. Hence please \textit{do not} use the \verb |appendix| command when writing an appendix at the end of your chapter or contribution. If there is only one the appendix is designated ``Appendix'', or ``Appendix 1'', or ``Appendix 2'', etc. if there is more than one.
%
%\begin{equation}
%a \times b = c
%\end{equation}

%\input{referenc}
\bibliographystyle{spmpsci}
\bibliography{network_geodesic_oslo}

\end{document}